\documentclass[12pt]{amsart}
\usepackage{amsmath,amsthm,amssymb,amscd,url,enumerate,mathtools}
\usepackage{graphicx}
\usepackage[margin=1in]{geometry}
\usepackage{afterpage}
\usepackage{tikz}
\usepackage[loadshadowlibrary]{todonotes}
\usepackage[all]{xy}
\usepackage[colorlinks=true,citecolor={blue},linkcolor = {purple},backref=page]{hyperref} 

\renewcommand*{\backrefalt}[4]{
    \ifcase #1 Not cited.%
          \or Cited on page~#2.%
          \else Cited on pages #2.%
    \fi} 

\usepackage{tabularx}

\usepackage{enumitem} 

\usepackage[symbol]{footmisc} 

\usepackage{caption}

\usepackage{comment}

\usepackage[square]{natbib}
\setcitestyle{numbers}

\def\namedlabel#1#2{\begingroup
    #2%
    \def\@currentlabel{#2}%
    \phantomsection\label{#1}\endgroup
}

\newcommand{\htodo}[2][inline]{\todo[linecolor=blue,backgroundcolor=blue!25,bordercolor=blue,#1,shadow]{\tiny #2}}
\newcommand{\jtodo}[2][inline]{\todo[linecolor=red,backgroundcolor=red!25,bordercolor=red,#1,shadow]{\tiny #2}}

\newcommand{\TITLE}{Critical Point Criteria and Dynamically Monogenic Polynomials}
\newcommand{\TITLERUNNING}{}

\theoremstyle{plain}
\newtheorem{theorem}{Theorem}

\newtheorem{lemma}[theorem]{Lemma}
\newtheorem{corollary}[theorem]{Corollary}

\newtheorem{condition}[theorem]{Condition}

\theoremstyle{definition}
\newtheorem{definition}[theorem]{Definition}

\theoremstyle{remark}
\newtheorem{remark}[theorem]{Remark}

\newtheorem{example}[theorem]{Example}
\newtheorem{question}[theorem]{Question}

\numberwithin{theorem}{section}

  {\end{list}}

  {\end{list}}

\newcommand{\tightoverset}[2]{%
  \mathop{#2}\limits^{\vbox to -.5ex{\kern-1.05ex\hbox{$#1$}\vss}}}

\numberwithin{equation}{section} 

\newcommand{\gm}{{\mathfrak{m}}}

\newcommand{\gp}{{\mathfrak{p}}}

\newcommand{\gP}{{\mathfrak{P}}}

\def\Ocal{{\mathcal O}}
\def\Pcal{{\mathcal P}}

\newcommand{\FF}{\mathbb{F}}

\newcommand{\QQ}{\mathbb{Q}}

\newcommand{\ZZ}{\mathbb{Z}}

\newcommand{\tensor}{\otimes}
\newcommand{\dnd}{\nmid}

\newcommand{\ol}[1]{\overline{#1}}
\newcommand{\inv}{^{-1}}

\newcommand{\Gal}{\operatorname{Gal}}
\newcommand{\Mon}{\operatorname{Mon}}

\newcommand{\Norm}{\operatorname{Norm}}

\title[\TITLERUNNING]{\TITLE}

\author[Joachim K\"onig]{Joachim K\"onig}
\address{Department of Mathematics Education\\
Korea National University of Education\\
250 Taeseongtabyeon-ro\\
Cheongju 28173, South Korea}
\email{jkoenig@knue.ac.kr}

\author[Hanson Smith]{Hanson Smith}
\address{Department of Mathematics\\
California State University San Marcos\\
333 S. Twin Oaks Valley Rd.\\
San Marcos, CA 92096}
\email{hsmith@csusm.edu}

\author[Zack Wolske]{Zack Wolske}
\address{Toronto, Ontario}
\email{zackwolske@gmail.com}


\begin{document}

\sloppy 


\baselineskip=17pt


\begin{abstract}
Let $K$ be a number field with ring of integers $\mathcal{O}_K$, and let $f(x)\in\mathcal{O}_K[x]$ be a monic, irreducible polynomial. We establish necessary and sufficient conditions in terms of the critical points of $f(x)$ for the iterates of $f(x)$ to be monogenic polynomials. More generally, we give necessary and sufficient conditions for the backwards orbits of elements of $\mathcal{O}_K$ under $f(x)$ to be monogenerators. We apply our criteria to construct novel examples of dynamically monogenic polynomials, yielding infinite towers of monogenic number fields with the backward orbit of one monogenerator giving a monogenerator at the next level.
\end{abstract}

\maketitle


\section{Introduction and main result}

Let $K$ be a number field with ring of integers $\Ocal_K$. We say $K$ is \textit{monogenic} if there exists some $\alpha$ in $\Ocal_K$ so that $\Ocal_K=\ZZ[\alpha]$. In other words, $K$ is monogenic if there exists some algebraic integer $\alpha$ such that the powers of $\alpha$ form a power integral basis $\{1,\alpha,\alpha^2,\cdots \alpha^{n-1}\}$ for $\Ocal_K$. In this case, we say that $\alpha$ is a \textit{monogenerator}, and we call the minimal polynomial of $\alpha$ a \textit{monogenic polynomial}. More generally, if $L$ is an extension of $K$, we say that \textit{$L$ is monogenic over $K$} if there exists an $\alpha\in\Ocal_L$ such that $\Ocal_L=\Ocal_K[\alpha]$. 

Questions regarding monogenicity\footnote{Also called `monogeneity' or `monogenity' in addition to some other less common terms.} are both theoretically and historically fundamental to algebraic number theory. Especially with the recent growth of arithmetic dynamics, it is natural to study the monogenicity of iterates of polynomials. 

In this work, we present succinct necessary and sufficient criteria for when the backward orbits of an algebraic integer under a monic, integral polynomial are monogenerators. More informally and less generally, we answer the question of when all iterates of a polynomial are monogenic. 

Before stating these criteria, we need a few definitions. Let $K$ be a number field and let $\alpha$ be integral and algebraic over $K$. If $\gp$ is a prime ideal of $\Ocal_K$, write $\Ocal_\gp$ to indicate $\Ocal_K$ adjoined the inverse of every element not in $\gp$, and write $k_\gp$ to indicate the residue field $\Ocal_K/\gp$. Let $L$ denote $K(\alpha)$. We say that $\Ocal_K[\alpha]$ is \textit{$\gp$-maximal} if $\Ocal_K[\alpha]\tensor_{\Ocal_K} \Ocal_\gp=\Ocal_L\tensor_{\Ocal_K} \Ocal_\gp$. When $K=\QQ$, it is equivalent to require that $\gp=p$ does not divide the index $\big[\Ocal_{\QQ(\alpha)}:\ZZ[\alpha]\big]$. If $\Ocal_K[\alpha]$ is $\gp$-maximal for each prime $\gp$, then $\Ocal_K[\alpha]=\Ocal_L$ and $\alpha$ is a monogenerator. If $\gP$ is a prime of $L$ above $\gp$, we write $e(\gP\mid \gp)$ for the ramification index. As usual, given a polynomial $f(x)$, we let $f^n(x)$ denote the $n$-fold composition, and for a prime $\gp$ of $\Ocal_K$, we let $v_\gp: K\to \mathbb{Z}\cup\{\infty\}$ denote the $\gp$-adic valuation.

Our main theorem allows a uniform treatment of all but finitely many primes, all of which are relatively small. We call these finitely many primes \textit{exceptional}. See Condition \ref{Cond: NonExceptional} for a precise definition.
In particular, $\gp$ is non-exceptional if $\operatorname{char} k_\gp > \deg(f)$ or if each irreducible factor of $f'(x)$ in $K[x]$ remains separable in $k_\gp[x]$. We are now able to state the main result:

\begin{theorem}[The Critical Point Criteria for Monogenicity]\label{Thm: MainRelativeOrbits}
Let $K$ be a number field. Fix $N > 0$. Suppose $f(x)\in \Ocal_K[x]$ has a leading coefficient that is a unit, and let $a\in \Ocal_K$ be such that $f^n(x)-a$ is irreducible over $K$ for $1\leq n\leq N$. Let $\gp$ be a prime of $\Ocal_K$. There exists a polynomial $g(x):=g_{\gp}(x)\in \Ocal_K[x]$ and a natural number $N_\gp$ such that the following holds: Let $M$ be an extension of $K$ where $g(x)$ splits and $\gP$ a prime of $M$ above $\gp$. Then the polynomial $f^n(x)-a$ is $\gp$-maximal for each $1\le n\le N$ if and only if for each root $\theta\in M$ of $g(x)$ and each $1\le n\le N_\gp$, one has $v_\gP(f^n(\theta)-a)\le e(\gP\mid \gp)$. More concretely,

\begin{enumerate}[label=\textbf{\arabic*.}, ref=\arabic*]

\item If $\gp$ is non-exceptional, one may choose $N_\gp = N$ and $g(x)=f'(x)$ (independently of $\gp$ and $a$). \label{Main1b.}

\item If $\gp$ is exceptional but is such that $f'(x)\not\equiv 0\bmod \gp$, one may choose $N_\gp = N$ and $g(x)=\prod_{i=1}^r \phi_i(x)$ (independently of $a$), where the $\phi_i(x)$ are lifts of the distinct irreducible factors of $f'(x)$ in $k_{\gp}[x]$ to $\Ocal_K[x]$. \label{Main1.}

\item If $\gp$ is such that $f'(x)\equiv 0\bmod \gp$, one may choose $g(x)$ fulfilling $f(x)-a = g(x)^p$ in $k_\gp[x]$ with $p = \operatorname{char}k_\gp$ and $N_{\gp}=1$. \label{Main2.}

\end{enumerate}
\end{theorem}

Theorem \ref{Thm: MainRelativeOrbits} will be proved  via Theorems \ref{Thm: MainLift} and \ref{Thm: MainNonExceptional}.
Applications of Theorem \ref{Thm: MainRelativeOrbits} become particularly striking in cases where (for $f(x)$ and $a$ suitably chosen) it can be applied simultaneously for all $n\in \mathbb{N}$, since in this case one obtains the simultaneous monogenicity of the polynomials $f^n(x)-a$ 
for all $n\in \mathbb{N}$. Such a pair $(f(x),a)$ is called a \textit{dynamically monogenic pair} (cf. Definition \ref{Def: DynMono}).
Investigating such pairs connects the topic of monogenicity with the quickly expanding field of arithmetic dynamics, a central problem of which is the study of arithmetic and Galois-theoretic properties of the backwards orbits $\cup_{j\in \mathbb{N}} f^{-j}(\{a\})$ of a value $a$ under a polynomial $f(x)$.

Theorem \ref{Thm: MainRelativeOrbits} and the applications herein generalize and subsume several previous results on simultaneous monogenicity of $f^n(x)-a$ for all $n\in \mathbb{N}$ for various $f(x)$ and $a$. This work was inspired by an effort to generalize \cite{SmithDynRadical} and \cite{SmithWolske} (and hence generalize \cite{Castillo} and \cite{Ruofan}). Theorem \ref{Thm: MainRelativeOrbits} can be seen as a generalization of other recent work including \cite{SharmaSarmaLaishram} (which focuses on $f(x) = x^m-a$) and \cite{LJones2021} (which considers specific families such as $f(x)=(x-a)^m+a$). \cite{GassertRadical} made an earlier connection between monogenicity and dynamics by giving conditions to guarantee the monogenicity of $T_\ell^n(x)-a$ when $T_\ell(x)$ is the Chebyshev polynomial of prime degree $\ell$ and $T_\ell^n(x)-a$ is irreducible. By applying Theorem \ref{Thm: MainRelativeOrbits} and some additional work on irreducibility, we obtain Theorem \ref{thm:cheby} which gives necessary and sufficient conditions for the dynamical monogenicity of $\big(T_\ell(x),a\big)$ with no degree restriction or irreducibility assumption. From another point of view, results on the monogenicity of radical polynomials can be readily generalized by taking a dynamical perspective and considering backward orbits under the power map $f(x)=x^d$. Thus the main result of \cite{SmithRadical} (and hence of \cite{GassertPowerMap}) can be readily obtained from Theorem \ref{Thm: MainRelativeOrbits}.

The generality of Theorem \ref{Thm: MainRelativeOrbits} also demonstrates some of the difficulties in identifying dynamically monogenic pairs $(f(x),a)$. Firstly, controlling $\gp$-adic valuations of infinitely many values simultaneously, as often required in Conditions \ref{Main1b.} and \ref{Main1.}, is very difficult, and unconditional positive results should often be expected to be out of reach. However, it may happen that the set of these values is actually a finite set. This occurs precisely when the polynomial $f(x)$ is post-critically finite. We therefore dedicate special attention to such polynomials in our applications in Section \ref{sec:appl}, notably extending previous results on specific PCF polynomials such as $f(x)=x^d$ and $T_\ell(x)$, as mentioned above. 
Moreover, the application of Theorem \ref{Thm: MainRelativeOrbits} for all $n\in \mathbb{N}$ requires a simultaneous irreducibility assumption for all polynomials $f^n(x)-a$. In many cases, proving this property is considered a difficult problem in arithmetic dynamics, and many previous studies have 
resorted to accepting it as a black-box assumption.
A key observation of Sections \ref{Sec: DynIrred} and \ref{sec:appl} is that, in several important key cases, this property follows from the combination of Conditions \ref{Main1b.} and \ref{Main2.} (cf.\ Theorems \ref{thm:unicrit} and \ref{thm:cheby}) or can be guaranteed for ``many" values $a$ (cf. the example given by Theorem \ref{Thm:ManyCritsExample}). 


\section{Establishing the critical point criteria}
\label{sec:critpoint}

We begin with a criterion for the monogenicity of a polynomial that is due to Uchida \cite{Uchida}. See \cite{VidauxVidela} for a comparison to other criteria for monogenicity.

\begin{theorem}[Uchida's Criterion]
Let $R$ be a Dedekind ring. Let $\alpha$ be an element of some integral domain which contains $R$, and let $\alpha$ be integral over $R$. Then $R[\alpha]$ is a Dedekind ring if and only if the defining polynomial $f(x)$ of $\alpha$ is not contained in $\gm^2$ for any maximal ideal $\gm$ of the polynomial ring $R[x]$.
\end{theorem}

To get a feel for Uchida's result, take $R=\ZZ$. Here $\gm = \big(p,\phi(x)\big)$ for some integral prime $p$ and some $\phi(x)\in\ZZ[x]$ with $\ol{\phi(x)}\in\FF_p[x]$ irreducible. Thus $\gm^2=\big(p^2, p\phi(x),\phi(x)^2\big)$. We can rephrase Uchida's criterion to require that $f(x)\neq \phi(x)^2g(x)+p\phi(x)h(x)+p^2j(x)$ for any $p$ and $\phi(x)$ as above and any $g(x),h(x),j(x)\in\ZZ[x]$. 

\begin{remark}\label{Remark: pmaximalandlocal}
Uchida's criterion is generally amicable towards localization. Indeed, let $K$ be the field of fractions of a Dedekind domain $R$. We will be most interested in the case where $K$ is a number field and $R=\Ocal_K$, or $R$ is $\Ocal_K$ localized at a prime ideal. 
Let $\alpha$ be integral and algebraic over $R$. If $\gp$ is a prime ideal of $R$, write $R_\gp$ to indicate $R$ localized at $\gp$. (By \textit{localized}, we mean that we adjoin the inverse of every element not in $\gp$.) When $R$ is $\Ocal_K$, then we will simply write $\Ocal_\gp$. Let $S$ be the integral closure of $R[\alpha]$. We say that $R[\alpha]$ is \textit{$\gp$-maximal} if $R[\alpha]\tensor_{R} R_\gp=S\tensor_{R} R_\gp$. When $R=\ZZ$, it is equivalent to require that $\gp=p$ does not divide $\big[\Ocal_{\QQ(\alpha)}:\ZZ[\alpha]\big]$. 
Uchida \cite[Lemma]{Uchida} proves that if $\gm$ is a maximal ideal of $R[x]$ and $\gm$ contains an integral polynomial, then $\gm$ is of the form $\big(\gp,\phi(x)\big)$ for some maximal ideal $\gp$ of $R$ and some $\phi(x)\in R[x]$ that is irreducible modulo $\gp$.
Let $f(x)$ be the minimal polynomial of $\alpha$ over $R$, (or, equivalently, over $R_\gp$), 
and let $\pi_\gp$ be a generator of $\gp$ in $R_\gp$. 
Now, $R[\alpha]$ is $\gp$-maximal if and only if $f(x)$ is not contained in the square of any maximal ideal of $R_\gp[x]$. That is, if and only if $f(x)\notin \big(\pi_\gp, \phi(x)\big)^2$ for any lift $\phi(x)\in R_\gp[x]$ of an irreducible in $R/\gp[x]$. Expanding, we see that $f(x)$ is \textit{$\gp$-maximal} if and only if for each $\phi(x)$ as above $f(x)\neq \phi(x)^2g(x)+\pi_\gp\phi(x)h(x)+\pi_\gp^2j(x)$ for any $g(x), h(x), j(x)\in R_\gp[x]$.
\end{remark}

Let $k_\gp$ be the residue field of $R$ at $\gp$. Since $f(x)$ and $f'(x)$ share a factor $\phi(x)$ if and only $\phi(x)$ is a repeated factor of $f(x)$, we have the following restatement of Uchida's Criterion.

\begin{theorem}[Restatement of Uchida's Criterion]\label{Thm: UchidaRestatement}
Let $f(x)\in R[x]$ be an irreducible polynomial with leading coefficient a unit, and let $\alpha$ be a root. 

If $f'(x)\not\equiv 0\bmod \gp$, then the ring $R[\alpha]$ is $\gp$-maximal if and only if for every factor $\ol{\phi(x)}$ of $f'(x)$ in $k_\gp[x]$ and any lift $\phi(x)\in R_\gp[x]$ one has $\phi(x)\nmid f(x)$ in $R_\gp/\pi_\gp^2[x]$. 

If $f'(x)\equiv 0 \bmod \gp$, then $R[\alpha]$ is $\gp$-maximal if and only if for every factor $\ol{\phi(x)}$ of $f(x)$ in $k_\gp[x]$ and any lift $\phi(x)\in R_\gp[x]$ one has $\phi(x)\nmid f(x)$ in $R_\gp/\pi_\gp^2[x]$. \end{theorem}

\begin{remark}
    Letting $R=\ZZ$ for the sake of being explicit, note that we cannot simply require that $f(x)$ and $f'(x)$ do not share a factor in $\ZZ/p^2\ZZ[x]$. If $f(x) = \phi(x)^2g(x)+p\phi(x)h(x)+p^2j(x)$, then we compute that
    \[f'(x)=2\phi(x)\phi'(x)g(x)+\phi(x)^2g'(x)+p\phi'(x)h(x)+p\phi(x)h'(x)+p^2j'(x).\]
    We see that the term $p\phi'(x)h(x)$ may prevent $f'(x)$ and $f(x)$ from sharing $\phi(x)$ as a factor modulo $p^2$.
\end{remark}

Henceforth, we will focus on the case where $R$ is $\Ocal_K$ or $\Ocal_\gp$, the localization of $\Ocal_K$ at a prime ideal $\gp$. From Remark \ref{Remark: pmaximalandlocal}, we see that in investigating the $\gp$-maximality of $\Ocal_K[\alpha]$ it is equivalent to work in $\Ocal_\gp[\alpha]$. 
In what follows, our exposition centers on $\Ocal_K$, though the equivalent statements hold for $\Ocal_\gp$. We will use this in later sections when we consider the $\gp$-maximalility of polynomials that are $\Ocal_\gp$-integral but not $\Ocal_K$-integral. 

Note that $f'(x)\equiv 0 \bmod \gp$ if and only if $f(x)=g(x)^p \bmod \gp$, where $p$ is the characteristic of $k_\gp$ and $g(x)$ is a polynomial in $k_\gp[x]$. As we can see, every element of $\ol{k_\gp}$ is a critical point and every root of $f(x)$ is a multiple root. We call such a prime $\gp$ a \textit{vanishing prime} for $f(x)$. 

On the other hand, suppose $f'(x)\not\equiv 0\bmod \gp$. In this case, we say $\gp$ is a \textit{non-vanishing prime}. 
With Theorem \ref{Thm: UchidaRestatement}, the following lemma relates $\gp$-maximality to lifts of factors of $f'(x)$ modulo $\gp$, hence to lifts of the critical points of $f(x)$ in $\ol{k_\gp}$.
    
    \begin{lemma}[Factors, Critical Lifts, and Valuations]\label{Lem: FactorstoCritPointsBetter}
    Suppose $f'(x)\not\equiv 0\bmod \gp$, and write $f'(x)\equiv \prod_{i=1}^r \phi_i(x)^{e_i}\bmod \gp$ for a factorization of $f'(x)$ into irreducibles in $k_\gp[x]$ where we have chosen lifts $\phi_i(x)$ of each irreducible factor to $\Ocal_K[x]$. Let $J$ be an extension of $K$ where each $\phi_i(x)$ splits completely, and let $\gP$ be a prime of $\Ocal_J$ above $\gp$. Then both of the following hold:
    \begin{enumerate}[label=\alph*), ref=\alph*)]
    \item If $\theta\in J$ fulfills $v_{\gP}(\phi(\theta))\ge e(\gP\mid \gp)$ \footnote{For example, but not necessarily, $\phi(\theta)=0$.} for a lift $\phi(x)$ of some irreducible factor of $f'(x)$ in $k_{\gp}[x]$ and  $v_\gP\big(f(\theta)\big)>e(\gP\mid \gp)$, then $\phi(x)$ divides $f(x)$ in $\Ocal_K/\gp^2[x]$.\label{Lemlifta}
    \item If conversely a lift $\phi(x)$ of an irreducible factor of $f'(x)$ in $k_\gp[x]$ divides $f(x)$ in $\Ocal_K/\gp^2[x]$, then $v_\gP\big(f(\theta)\big)>e(\gP\mid \gp)$ for every $\theta\in J$ such that $v_\gP\big(\phi(\theta)\big)\geq e(\gP\mid \gp)$. \label{Lemliftb}
    \end{enumerate}
\end{lemma}


 \begin{proof}
    Let $\theta\in J$ be such that $v_{\gP}\big(\phi(\theta)\big)\geq e(\gP\mid \gp)$ and $v_\gP\big(f(\theta)\big)>e(\gP\mid \gp)$. We see that $\phi(\theta), f(\theta)\equiv 0\bmod \gP$. 
        Since $k_\gP=\Ocal_J/\gP$ is an extension of $k_\gp$, we see that $\gcd\big(f(x),\phi(x)\big)\neq 1$ in $k_\gp[x]$. Since $\phi(x)$ is irreducible in $k_{\gp}[x]$, we note $\phi(x)$ divides $f(x)$ in $k_\gp[x]$. However, since $\phi(x)$ also divides $f'(x)$ in $k_\gp[x]$, we see that $\phi(x)^2$ divides $f(x)$ in $k_\gp[x]$. Thus we can write $f(x)=\phi(x)^2g(x)+bh(x)$ in $\Ocal_K[x]$ for some $b\in \gp$ and some $h(x)\in \Ocal_K[x]$. We evaluate at $\theta$ to obtain $f(\theta)=\phi(\theta)^2 g(\theta)+ bh(\theta)$. Noting $v_\gP\big(\phi(\theta)^2 g(\theta)\big)\geq 2 e(\gP\mid \gp)$, we find $v_\gP\big(bh(\theta)\big)>e(\gP\mid \gp)$. Thus either $b\in \gp^2$ or $v_\gP\big(h(\theta)\big)>0$. In the former case, our desired result is clear. 
    In the latter case, we note $\gcd\big(h(x),\phi(x)\big)\neq 1$ in $k_\gP[x]$ and hence in $k_\gp[x]$ as above. Thus, $h(x)=\phi(x)j(x)+c\ell(x)$ with $c\in \gp\Ocal_K$. Therefore,
    \[f(x)=\phi(x)^2g(x)+bh(x)=\phi(x)^2g(x)+b\big(\phi(x)j(x)+c\ell(x)\big),\]
    and we have \ref{Lemlifta}. 
    
    Suppose conversely that $\phi(x)$ divides $f(x)$ in $\Ocal_K/\gp^2[x]$. Since $\phi(x)$ is a factor of $f'(x)$ modulo $\gp$, we have $f(x)=\phi(x)^2g(x)+b\phi(x)h(x)+cj(x)$ where $b\in \gp$ and $c\in \gp^2$. Evaluating at any $\theta$ such that $v_\gP\big(\phi(\theta)\big)\geq e(\gP\mid \gp)$, we see that $v_\gP\big(f(\theta)\big)\geq 2e(\gP\mid \gp)$, yielding \ref{Lemliftb}.
    \end{proof}

When $\gp$ is a vanishing prime, we have $f(x)\equiv g(x)^p\bmod \gp$ for some $g(x)\in\Ocal_K[x]$. The next lemma connects roots of $g(x)$ to factors of $f(x)$.

\begin{lemma}[Vanishing Primes and Roots]\label{Lem: VanishingPrimesandRoots}
   Let $\gp$ be a vanishing prime of residue characteristic $p$, so $f(x)\equiv g(x)^p\bmod \gp$ for some $g(x)\in\Ocal_K[x]$. Let $L$ be an extension of $K$ where $g(x)$ splits, and let $\gP$ be a prime of $L$ above $\gp$. 
   Then there is a lift $\phi(x)$ of a repeated irreducible factor of $f(x)$ in $k_\gp[x]$ that divides $f(x)$ in $\Ocal_K/\gp^2[x]$ if and only if $v_\gP\big(f(\tau)\big)>e(\gP\mid \gp)$, where $\tau\in\Ocal_L$ is a root of $g(x)$. 
\end{lemma}

 \begin{proof}
    Since $\phi(x)$ is irreducible in $k_\gp[x]$ and since $f(x)\equiv g(x)^p\bmod \gp$, we see that $\phi(x)$ divides $g(x)$ in $k_\gp[x]$. Thus, we can write $g(x)\equiv \phi(x)^mr(x)\bmod \gp$ with $\gcd\big(\phi(x),r(x)\big)=1$ in $k_\gp[x]$. Since $g(x)$ splits in $L$, we see there is an algebraic integer $\tau$ that is a root of $g(x)$ and such that $\phi(\tau)\equiv 0\bmod\gP$. We have $0=g(\tau)\equiv \phi(\tau)^m r(\tau)\bmod \gp\Ocal_L$. Since $\gcd\big(\phi(x),r(x)\big)=1$ and since $v_\gP\big(\phi(\tau)\big)>0$, we see $v_\gP\big(r(\tau)\big)=0$ and $v_\gP\big(\phi(\tau)^m\big)\geq e(\gP\mid \gp)$. 

    Since $\phi(x)$ divides $f(x)$ modulo $\gp^2$ and since $f(x)\equiv g(x)^p\bmod \gp$, we have $f(x)=\big(\phi(x)^mr(x)\big)^p+bh(x)$ with $b\in\gp$ and either $b\in\gp^2$ or $\phi(x)$ dividing $h(x)$ modulo $\gp$. Evaluating at $\tau$ and noting $v_\gP\big(\phi(\tau)^m\big)\geq e(\gP\mid \gp)$, 
    we have our desired result.

    For the converse, let $\tau\in L$ be a root of $g(x)$ such that $v_\gP\big(f(\tau)\big)>e(\gP\mid \gp)$. 
    Let $\phi(x)$ be an irreducible factor of $g(x)$ in $\Ocal_K[x]$ such that $\phi(\tau)\equiv 0\bmod \gP$. (Note that we can assume that the number of irreducible factors of $g(x)$ in $k_\gp[x]$ is equal to that in $\Ocal_K[x]$.)
   
    By hypothesis, $f(x)=g(x)^p+bh(x)$ with $b\in\gp$ and $h(x)\in\Ocal_K[x]$. Evaluating at $\tau$, we see that $v_\gP\big(bh(\tau)\big)>e(\gP\mid \gp)$. Thus either $b\in\gp^2$ or $\gcd\big(\phi(x),h(x)\big)>1$ in $k_\gp[x]$. In either case we have our result.
    \end{proof}

The lemmas above yield a rephrasing of Uchida's criterion in terms of lifts of factors of $f'(x)\bmod \gp$.

\begin{theorem}[Uchida's Critical Lift Criterion]\label{Thm: RelCritPointCrit}
Let $f(x)\in\Ocal_K[x]$ be an irreducible polynomial with leading coefficient a unit, and let $\alpha$ be a root. 
For $\gp\subset \Ocal_K$ a non-vanishing prime ($f'(x)\not\equiv 0\bmod \gp$), write $f'(x)\equiv \prod_{i=1}^r \phi_i(x)^{e_i}\bmod \gp$ for a factorization of $f'(x)$ into irreducibles in $k_\gp[x]$ where we have chosen lifts $\phi_i(x)$ of each irreducible factor to $\Ocal_K[x]$. Let $J$ be an extension of $K$ where each $\phi_i(x)$ splits completely, and let $\gP$ be a prime of $\Ocal_J$ above $\gp$.
Then $\Ocal_K[\alpha]$ is $\gp$-maximal if and only if for each $i$ one has $v_\gP\big(f(\theta_i)\big)\leq e(\gP\mid \gp)$ where $\theta_i$ is an arbitrary element of $J$ such that $v_\gP\big(\phi_i(\theta_i)\big)\geq e(\gP\mid \gp)$.

If $\gp$ is a vanishing prime, then we can write $f(x)\equiv g(x)^p\bmod \gp$ for some $g(x)\in \Ocal_K[x]$. Let $L$ be an extension of $K$ where $g(x)$ splits, and let $\gP$ be a prime of $\Ocal_L$ above $\gp$. 
In this case, the ring $\Ocal_K[\alpha]$ is $\gp$-maximal if and only if for each $\tau\in \Ocal_L$ such that $g(\tau)=0$, one has $v_\gP\big(f(\tau)\big)\leq e(\gP\mid \gp)$. 
\end{theorem}


\begin{proof}
    Combining Lemmas \ref{Lem: FactorstoCritPointsBetter} and \ref{Lem: VanishingPrimesandRoots} with Theorem \ref{Thm: UchidaRestatement}, we have the result.
\end{proof}


We now proceed to considering monogenicity of not just one fixed polynomial $f(x)$ but of all iterates $f^n(x)-a$, $n\in \mathbb{N}$ simultaneously (for a fixed $a\in \Ocal_K$). The following is the first version of our main criterion.

\begin{theorem}[The Critical Lift Criteria for Monogenicity]\label{Thm: MainLift}
Let $K$ be a number field. Fix $N>0$. Suppose $f(x)\in \Ocal_K[x]$ has a leading coefficient that is a unit, and let $a\in\Ocal_K$ be such that $f^n(x)-a$ is irreducible over $K$ for $1\leq n\leq N$. 

\begin{enumerate}[label=\arabic*., ref=\arabic*]
\item If $\gp$ is not a vanishing prime (i.e., $f'(x)\not\equiv 0\bmod \gp$), then we may write then we may write $f'(x) \equiv \prod_{i=1}^r\phi_i(x)^{e_r}\bmod \gp$ where each $\phi_i(x)$ is a lift of a distinct irreducible factor of $f'(x)$ in $k_\gp[x]$ to $\Ocal_K[x]$. 
Let $J$ be an extension of $K$ where each $\phi_i(x)$ splits and let $\gP$ be a prime of $J$ above $\gp$. Then the polynomial $f^n(x)-a$ is $\gp$-maximal for each $1\leq n\leq N$ if and only if for each $\theta$ in $J$ that is a root 
of some $\phi_i(x)$ one has $v_\gP\big(f^n(\theta)-a\big)\leq e(\gP\mid \gp)$ for each $1\leq n\leq N$.
\label{MainLiftNonVan}

\item If $\gp$ is a vanishing prime (i.e., $f'(x)\equiv 0\bmod \gp$), then we have $f(x)-a\equiv g(x)^p\bmod \gp$ for some $g(x)\in \Ocal_K[x]$. 
Let $L$ be an extension of $K$ where $g(x)$ splits. Let $\gP$ be a prime of $\Ocal_L$ above $\gp$ with ramification index $e(\gP\mid \gp)$. The polynomial $f^n(x)-a$ is $\gp$-maximal for all $1\leq n \leq N$ if and only if for every root $\tau$ of $g(x)$ in $\Ocal_L$, one has $v_\gP\big(f(\tau)-a\big)\leq e(\gP\mid \gp)$.\label{MainLiftVan}
\end{enumerate}

\end{theorem}

\begin{proof}
First, we focus on the non-vanishing primes. Let $\gp\subset \Ocal_K$ be such a prime. 
Though there is a natural proof via induction, we proceed directly. Note Lemma \ref{Lem: FactorstoCritPointsBetter} and our hypothesis immediately imply that no $\phi_i(x)$ divides $f^n(x)-a$ in $\Ocal_K/\gp^2[x]$ for any $1\leq n\leq N$. Fix an $n$ in the range $1\leq n\leq N$. We wish to show $f^n(x)-a$ is $\gp$-maximal. We have
\begin{equation}\label{Eq: Relchainrule}
\frac{d}{dx}\Big(f^{n}(x)-a\Big)=f'\big(f^{n-1}(x)\big)\cdots f'(x)\equiv \prod_{j=1}^s \gamma_{j}(x)^{m_j} \bmod \gp,
\end{equation}
where the $\gamma_j(x)$ are lifts of the irreducible factors of $\frac{d}{dx}\big(f^{n}(x)-a\big)$ in $k_\gp[x]$ to $\Ocal_K[x]$. Let $J'$ be an extension of $J$ where all the $\gamma_j(x)$ split, and let $\Pcal$ be a prime of $J'$ above $\gP$.

For a contradiction, suppose $f^n(x)-a$ is not $\gp$-maximal. Thus, Theorem \ref{Thm: RelCritPointCrit} shows there is some $\tau\in J'$ and some $\gamma_j(x)\in \Ocal_K[x]$ such that $v_\Pcal\big(\gamma_j(\tau)\big)\geq e(\Pcal\mid \gp)$ and $v_\Pcal\big(f^n(\tau)-a\big) > e(\Pcal\mid \gp)$. Since $\gamma_j(x)$ is irreducible in $k_\gp[x]$, we see that $\gamma_j(x)$ divides $f'\big(f^t(x)\big)$ in $k_\gp[x]$ for some $0\leq t\leq n-1$. Hence $\gamma_j(x)$ divides $\phi_i\big(f^t(x)\big)$ in $k_\gp[x]$ for some $i$. Since $v_\Pcal\big(\gamma_j(\tau)\big)\geq e(\Pcal\mid \gp)$, we have $v_\Pcal\big(\phi_i\big(f^t(\tau)\big)\big)\geq e(\Pcal\mid \gp)$. As 
\[v_\Pcal\big(f^n(\tau)-a\big) = v_\Pcal\Big(f^{n-t}\big(f^t(\tau)\big)-a\Big) > e(\Pcal\mid \gp),\]
Lemma \ref{Lem: FactorstoCritPointsBetter} shows that $\phi_i(x)$ divides $f^{n-t}(x)-a$ in $\Ocal_K/\gp^2[x]$, yielding the desired contradiction. We have established \ref{MainLiftNonVan}.

Now we turn to the vanishing primes. Suppose $f'(x)\equiv 0\bmod \gp$, so $f(x)-a\equiv g(x)^p\bmod \gp$ for some $g(x)\in \Ocal_K[x]$. Let $L$ and $\gP$ be as in the theorem statement. 
For the monogenicity of $f^n(x)-a$, Theorem \ref{Thm: RelCritPointCrit} shows it is necessary that $v_\gP\big(f(\tau)-a\big)\leq e(\gP\mid \gp)$ for each root $\tau$ of $g(x)$. Suppose this is the case and consider $f^n(x)-a$ with $1\leq n \leq N$. Since $f'(x)\equiv 0\bmod \gp$, we compute
\[\frac{d}{dx}\Big(f^{n}(x)-a\Big)=f'\big(f^{n-1}(x)\big)f'\big(f^{n-2}(x)\big)\cdots f'(x)\equiv 0\bmod \gp.\]
Thus $\gp$ is a vanishing prime for $f^{n}(x)-a$, and $f^{n}(x)-a\equiv g\big(f^{n-1}(x)\big)^p\bmod \gp$. Let $L'$ be an extension of $L$ where $g\big(f^{n-1}(x)\big)$ splits completely and let $\Pcal$ be a prime of $L'$ above $\gP$. Theorem \ref{Thm: RelCritPointCrit} shows that $f^n(x)-a$ is $\gp$-maximal if and only if for any $\tau'\in L'$ that is a root of $g\big(f^{n-1}(x)\big)$, one has $v_\Pcal\big(f^n(\tau')-a\big)\leq e(\Pcal\mid \gp)$.
For a contradiction, suppose $\tau'$ is a root of $g\big(f^{n-1}(x)\big)$ in $L'$ with $v_\Pcal\big(f^n(\tau')-a\big) > e(\Pcal\mid \gp)$. Now $f^{n-1}(\tau')$ is a root of $g(x)$ in $L$. We have $v_\Pcal\big(f^n(\tau')-a\big)=e(\Pcal\mid\gP)v_\gP\big(f(f^{n-1}(\tau'))-a\big)$ and $e(\Pcal\mid\gp)=e(\Pcal\mid\gP)e(\gP\mid\gp)$, so $v_\gP\big(f(f^{n-1}(\tau'))-a\big) > e(\gP\mid \gp)$. This is a contradiction. We conclude that $v_\Pcal\big(f^n(\tau')-a\big) \leq e(\Pcal\mid \gp)$ and $f^n(x)-a$ is $\gp$ maximal for all $1\leq n\leq N$.
\end{proof}


\begin{remark} \label{Rem: CondiesnotIrred} 
Note that irreducibility is not necessarily automatic from Conditions \ref{MainLiftNonVan} and \ref{MainLiftVan} of Theorem \ref{Thm: MainLift}. Hence in general, the assumption of irreducibility cannot be dropped. An example is $f(x)=(x^2-2)(x^2-3)$, with $K=\mathbb{Q}$, $a=0$, $N=1$. The mod $2$ roots $0$ and $1$ evaluate to $2 \bmod 4$. Hence Condition \ref{MainLiftVan} holds for the only vanishing prime $2$. Moreover, $f'(x)=2x(2x^2-5)$, yielding the critical values are $6$ and $-1/4$. These have squarefree numerators, meaning that Condition \ref{MainLiftNonVan} holds for all non-vanishing primes.
\end{remark}

\begin{remark}\label{Rem: FiniteCheckLift}
Since for arithmetic dynamics applications, we intend to use results such as Theorem \ref{Thm: MainLift} for all $n\in \mathbb{N}$ simultaneously, it is worth noting the following: for any fixed prime $\gp$ of $\Ocal_K$, assumed to be non-vanishing for $f(x)$, there exists a bound $N_\gp\in \mathbb{N}$ such that Condition \ref{MainLiftNonVan} of Theorem \ref{Thm: MainLift} holds for all $n\in \mathbb{N}$ as soon as it holds for all $n\le N_\gp$. Indeed, since the field generated by all roots of $f'(x)\bmod \gp$ is a finite extension of $k_\gp$ of degree less than $\deg(f)$, the sequence $\{f^n(\tau)\mid n\in \mathbb{N}\}$ is preperiodic modulo $\gP^{e+1}$ with $\gP$ as in Condition \ref{MainLiftNonVan} and $e:=e(\gP\mid \gp)$. But, as soon as $f^n(\tau)\equiv f^m(\tau) \bmod \gP^{e+1}$, then the properties $v_\gP(f^n(\tau)-a)\le e$ and $v_\gP(f^m(\tau)-a)\le e$ are equivalent.
\end{remark}


While Theorem \ref{Thm: MainLift} is already helpful for studying the monogenicity of polynomials and their iterates, the dependence of irreducible factors $\gamma_j(x)$ on the respective prime $\gp$ is not ideal. Instead, we hope to work with the roots of $f'(x)$ in $\ol{K}$ (i.e., the critical points of $f(x)$) independent of $\gp$. The following lemma is useful in determining when this can be achieved. We employed the following notation: If $g(x)$ and $\phi(x)$ are polynomials with images in $k_\gp[x]$, then we write $v_{\phi,\gp}\big(g(x)\big)$ for the multiplicity to which $\phi(x)$ divides $g(x)$ in $k_\gp[x]$. 

\begin{lemma}[Factors, Critical Points, and Valuations]\label{Lem: FactorstoCritPoints}
    Let $h(x)$ be an irreducible factor of $f'(x)$ in $K[x]$. Suppose $\phi(x)$ is a lift of an irreducible factor of $h(x)$ in $k_\gp[x]$ to $\Ocal_K[x]$, and suppose either $v_{\phi,\gp}\big(f(x)\big)=0$ or $v_{\phi,\gp}\big(f(x)\big)>v_{\phi,\gp}\big(h(x)\big)$. 
    Let $M$ be an extension of $K$ where $f'(x)$ splits completely, and let $\gP$ be a prime of $\Ocal_M$ above $\gp$. Then $\phi(x)$ divides $f(x)$ in $\Ocal_K/\gp^2[x]$ if and only if $v_\gP\big(f(\theta)\big)>e(\gP\mid \gp)$ for some $\theta$ in $M$ that is a root of $h(x)$ such that $\phi(\theta)\equiv 0\bmod \gP$. In particular, if all the irreducible factors of $f'(x)$ in $K[x]$ satisfy the multiplicity hypothesis above, then $f(x)$ is $\gp$-maximal if and only if $v_\gP\big(f(\theta)\big)\leq e(\gP\mid \gp)$ for each critical point $\theta\in M$.
\end{lemma}

\begin{proof}
    For ease of notation, let $k=v_{\phi,\gp}\big(h(x)\big)$ and $m=v_{\phi,\gp}\big(f(x)\big)$. If $v_{\phi,\gp}\big(f(x)\big)=0$, then the result is clear, so suppose $\phi(x)$ divides $f(x)$ in $\Ocal_K/\gp^2[x]$. Let $\theta$ in $M$ be a critical point such that $\phi(\theta)\equiv 0\bmod \gP$. We have $h(x)=\phi(x)^kr(x) + d$ for some $r(x)\in\Ocal_K[x]$ that is relatively prime to $\phi(x)$ in $k_\gp[x]$ and some $d\in \gp\Ocal_K[x]$. Since $h(\theta)=0$, we see that $v_\gP\big(\phi(\theta)^k\big)\geq e(\gP\mid \gp)$. We can write $f(x)=\phi(x)^mg(x)+b\phi(x)+c$ with $b\in\gp\Ocal_K[x]$ and $c\in\gp^2\Ocal_K[x]$. Since $m>k$, we evaluate at $\theta$ to obtain $v_\gP\big(f(\theta)\big)>e(\gP\mid \gp)$.

    Conversely, suppose $v_\gP\big(f(\theta)\big)>e(\gP\mid \gp)$ for some critical point $\theta$ in $M$ such that $\phi(\theta)\equiv 0\bmod \gP$. We have $h(x)=\phi(x)^kr(x) + d$ for some $r(x)\in\Ocal_K[x]$ that is relatively prime to $\phi(x)$ in $k_\gp[x]$ and some $d\in \gp\Ocal_K[x]$. Since $h(\theta)=0$, we see that $v_\gP\big(\phi(\theta)^k\big)\geq e(\gP\mid \gp)$. We write $f(x)=\phi(x)^mg(x)+bj(x)+c$ with $b\in \gp$ and $c\in \gp^2\Ocal_K[x]$. Since $v_\gP\big(\phi(\theta)^mg(\theta)\big)>e(\gP\mid \gp)$ and $v_\gP\big(f(\theta)\big)>e(\gP\mid \gp)$, we see $v_\gP\big(bj(\theta)\big)>e(\gP\mid \gp)$. By a now familiar argument, $\phi(x)$ divides $f(x)$ in $\Ocal_K/\gp^2[x]$.

    The last statement is obtained by applying Theorem \ref{Thm: UchidaRestatement} and noting that any repeated factor of $f(x)$ in $k_\gp[x]$ will be a factor of $f'(x)$ in $k_\gp[x]$.
\end{proof}


As we will see in Theorem \ref{Thm: MainNonExceptional},
the conclusion of Theorem \ref{Thm: MainLift} for a non-vanishing prime 
$\gp$ remains valid if we replace the roots of the lifts $\prod_{i=1}^r \phi_i(x)^{e_i}\equiv f'(x)\bmod \gp$ by the global roots of $f'(x)$ itself as soon as the following condition holds:

\begin{condition}\label{Cond: NonExceptional}
    For every irreducible factor $h(x)$ of $f'(x)$ in $K[x]$, every irreducible factor $\phi(x)$ of $h(x)$ in $k_\gp[x]$, and every $n\in \mathbb{N}$ such that $v_{\phi(x)}(f^n(x)-a)>0$, we have $v_{\phi(x)}(f^n(x)-a)>v_{\phi(x)}(h(x))$. 
\end{condition}


Let us call a prime $\gp$ of $\Ocal_K$ a {\it non-exceptional} prime for $(f(x),a)$, if $f'(x)\not\equiv 0 \bmod \gp$ and furthermore Condition \ref{Cond: NonExceptional} holds for the given $(f(x),a)$. If this condition holds for all $a\in \Ocal_K$, we call $\gp$ a non-exceptional prime for $f(x)$. Naturally, we call $\gp$ \textit{exceptional}, for $(f(x),a)$ or $f(x)$, if these conditions do not hold.

We collect some important observations facilitating the verification of non-exceptionality in important cases. 
\begin{lemma}\label{Lem: Exceptional<=deg}
    Let $\gp$ be a prime of $K$ with residue characteristic $p$. Then the following hold:
    \begin{enumerate}[label=\alph*), ref=\alph*)]
    \item If the cover $\mathbb{P}^1_{\overline{\mathbb{F}_p}}\to \mathbb{P}^1_{\overline{\mathbb{F}_p}}$ given by $x\mapsto f(x)$ is tamely ramified, then $p$ is non-exceptional for $f(x)$. In particular, this holds whenever $p>\deg(f)$.\label{a)ramcover}
    \item If every irreducible factor of $f'(x)$ is separable modulo $\gp$, then $\gp$ is non-exceptional for $f(x)$. \label{b)sepfactors}
    \end{enumerate}
\end{lemma}
\begin{proof}
    To show \ref{a)ramcover}, it suffices to prove that  $f^n(x)-a$ has no repeated factors of multiplicity divisible by $p$ in $k_\gp[x]$ for any $n\in \mathbb{N}$, since indeed then the multiplicity of any irreducible factor decreases strictly in $\frac{d}{dx}(f^n(x))$ and thus a fortiori in $f'(x)$. Since the ramification index of a place $x\mapsto \alpha\in \overline{\mathbb{F}_p}$ under the cover $x\mapsto f(x)$ is simply the multiplicity of $x-\alpha$ in $f(x)-f(\alpha)$, tame ramification translates to saying that $f(x)-a$ has no root of multiplicity divisible by $p$ for any $a\in \overline{\mathbb{F}_p}$. This shows the claim for $n=1$. For arbitrary $n$, the claim simply follows by multiplicativity of ramification indices under composition of covers.

 \ref{b)sepfactors} follows from the definition: Let $h(x)$ be an irreducible factor of $f'(x)$ in $K[x]$. A simple factor of $h(x)$ that is a common factor of $f'(x)$ and $f^n(x)-a$ must necessarily be a repeated factor of $f^n(x)-a$. 
\end{proof}

The definition of non-exceptionality implies the following simplification of Theorem \ref{Thm: MainLift} for non-exceptional primes. 

\begin{theorem}[The Critical Point Criterion for Non-Exceptional Primes]\label{Thm: MainNonExceptional}
Let $K$ be a number field. Fix $N>0$. Suppose $f(x)\in \Ocal_K[x]$ has a leading coefficient that is a unit, and let $a\in\Ocal_K$ be such that $f^n(x)-a$ is irreducible over $K$ for $1\leq n\leq N$. Let $M$ be an extension of $K$ where $f'(x)$ splits. 
Let $\gp$ be a prime of $\Ocal_K$ which is non-exceptional for $(f(x),a)$, and let $\gP$ be a prime of $M$ above $\gp$. Then the polynomial $f^n(x)-a$ is $\gp$-maximal for each $1\leq n\leq N$ if and only if for each $1\leq n\leq N$ and each critical point $\theta\in M$, one has 
$v_\gP\big(f^n(\theta)-a\big) \leq e(\gP\mid \gp)$. \label{MainNonExceptional}
\end{theorem}

\begin{proof}
    Let $1\leq n\leq N$ and consider $\frac{d}{dx}\big(f^{n}(x)-a\big)=f'\big(f^{n-1}(x)\big)\cdots f'(x).$
    Let $M'$ be an extension of $M$ where $\frac{d}{dx}(f^n(x)-a)$ splits, and let $\Pcal$ be a prime of $M'$ above $\gP$. If $\tau\in M'$ is a critical point of $f^n(x)-a$, then $\tau$ is a root of $f'\big(f^t(x)\big)$ for some $0\leq t\leq n-1$. Thus $f^t(\tau)$ is a root of $f'(x)$. From this and the multiplicative nature of valuations, we see $v_{\Pcal}\big(f^n(\tau)-a\big)>e(\Pcal\mid \gp)$ for some $\tau$ that is a critical point of $f^n(x)-a$ if and only if $v_\gP\big(f^{n-t}(\theta)-a\big) > e(\gP\mid \gp)$ for some $\theta=f^t(\tau)$ that is a critical point of $f(x)-a$. Lemma \ref{Lem: FactorstoCritPoints} shows $v_{\Pcal}\big(f^n(\tau)-a\big)>e(\Pcal\mid \gp)$ for some $\tau$ that is a critical point of $f^n(x)-a$ if and only if $f^n(x)-a$ is not $\gp$-maximal.
\end{proof}   

Combining Theorem \ref{Thm: MainNonExceptional} together with Theorem \ref{Thm: MainLift} to deal with the finitely many exceptional primes yields a proof of Theorem \ref{Thm: MainRelativeOrbits}. 

\begin{remark}\label{Rem: MonPolyMonField2}
Theorem \ref{Thm: MainRelativeOrbits} gives necessary and sufficient conditions for the iterates of any monic polynomial $f(x)\in\Ocal_K[x]$ to be monogenic; however, even when a polynomial fails to be monogenic, the field generated by a root may be monogenic. For a simple example consider $f(x)=x^2$ and let $a\in\ZZ^{>1}$. We see that $f^n(x)+a^{2^n}=x^{2^n}+a^{2^n}$ is not monogenic. It fails to be $p$-maximal at all primes dividing $a$. However, $f^n(x)+a^{2^n}$ generates the same field as $x^{2^n}+1$, and the $2$-power cyclotomic fields are monogenic with monogenerator $\zeta_{2^{n+1}}$. Indeed, Theorem \ref{Thm: MainNonExceptional} gives a short proof: Each odd prime $p$ is a non-exceptional prime, and $v_p(0+1)=0$. The prime 2 is a vanishing prime, and $v_2(1^2+1)=1$.
\end{remark}

\begin{remark}\label{Rem: MonPolyMonFieldDedekind}
The assertion of Theorem \ref{Thm: MainNonExceptional} no longer holds for (non-vanishing) exceptional primes $\gp$. Indeed take Dedekind's example \cite{Dedekind}. Let $f(x)=x^3 - x^2 - 2x - 8\in \mathbb{Z}[x]$, $a=0$, and $p=2$. We have $f'(x)\equiv x^2\bmod 2$, and evaluating $v_2\big(f(0)\big)=v_2(-8) = 3 > 1$, we see from Theorem \ref{Thm: MainLift} that $f(x)$ is not $2$-maximal. On the other hand, the global roots of $f'(x)=3x^2-2x-2$ are $\theta = \frac{1\pm \sqrt{7}}{3}$. Choosing $\gP$ as the (unique, since ramified) prime extending $2$ in $\mathbb{Q}(\sqrt{7})$, we get $v_\gP\big(f(\theta)\big) = v_\gP\big(\frac{1}{27}(\pm 14\sqrt{7} - 236)\big) = 2 = e(\gP\mid 2)$, showing that an attempted application of Theorem \ref{Thm: MainNonExceptional} with the global roots of $f'(x)$ would lead to an erroneous result.
\end{remark}

For many of the examples we work with we do not need the field extensions $J/K$ and $L/K$ in Theorem \ref{Thm: MainLift} or $M/K$ in Theorem \ref{Thm: MainNonExceptional}. 
The following corollary captures this simplified situation. Note that we will assume here and in similar scenarios in the following sections that all non-vanishing primes are non-exceptional, an assumption that will be fulfilled in all of our applications. This assumption is, however, only made for sake of simplicity. Without it, finitely many further $K$-rationality conditions would have to be added to accommodate the finitely many (by Remark \ref{Rem: FiniteCheckLift}) relevant values $n\in \mathbb{N}$ for the finitely many exceptional non-vanishing primes.

\begin{corollary}\label{Cor: CritsinK}
    With the hypotheses as in Theorems \ref{Thm: MainLift} and \ref{Thm: MainNonExceptional}, suppose that 
    $f(x)$ has $K$-rational critical values, i.e., $f(\theta)\in K$ for each critical point $\theta$ of $f(x)$.
    Suppose further that every exceptional prime for $f(x)$ is a vanishing prime. For each such prime $\gp$, write $f(x)-a\equiv g(x)^p \bmod \gp$ 
    where $g(x)\in \Ocal_K[x]$, and suppose that for each root $\tau$ of $g(x)$ it holds that $f(\tau)\in \Ocal_K$.
    Then, $f^n(x)-a$ is monogenic for all $n\leq N$ if and only if both of the following conditions hold:
\begin{enumerate}[label=\textbf{\arabic*.}, ref=\arabic*]
        \item (Non-Vanishing Primes) If $f'(x)\not\equiv 0\bmod \gp$, then $v_\gp\big(f^n(\theta)-a\big)\leq 1$ for each critical point $\theta$ and each $n\leq N$. \label{MainCor1.}
        \item (Vanishing Primes) If $f'(x)\equiv 0\bmod \gp$, then $v_\gp\big(f(\tau)-a\big)\leq 1$ for each root $\tau$ of $g(x)$. \label{MainCor2.}
\end{enumerate}
\end{corollary}

\begin{proof}
    For the non-vanishing primes we let $M$ be an extension where $f'(x)$ splits and let $\gP$ be a prime of $M$ above $\gp$. Since $f(\theta)\in\Ocal_K$ for each critical point $\theta$, we have $v_\gP\big(f^n(\theta)-a\big)\leq e(\gP\mid \gp)$ if and only if $v_\gp\big(f^n(\theta)-a\big)\leq1$.
    Condition \ref{MainCor1.} follows from Theorem \ref{Thm: MainNonExceptional}.

    For the vanishing primes, we let $L$ be an extension where $g(x)$ splits and let $\gP$ be a prime of $L$ above $\gp$. Since $f(\tau)\in\Ocal_K$ for each root $\tau$ of $g(x)$ , we have $v_\gP\big(f(\tau)-a\big)\leq e(\gP\mid \gp)$ if and only if $v_\gp\big(f(\tau)-a\big)\leq1$. Condition \ref{MainCor2.} follows from Theorem \ref{Thm: MainLift}.
\end{proof}

\begin{definition}\label{Def: DynMono} We say that a polynomial $f(x)\in\Ocal_K[x]$ is \textit{dynamically monogenic} if $f^n(x)$ is monogenic for each $n\geq 1$. More generally, we say that a pair $f(x)\in \Ocal_K[x]$ and $a\in \Ocal_K$ is a \textit{dynamically monogenic pair} if $f^n(x)-a$ is a monogenic polynomial for each $n\geq 1$. For example, classical results show that for every prime $p$, the pair $(x^p, \zeta_p)$ is a dynamically monogenic pair over $\QQ(\zeta_p)$. Note that some authors use the term \textit{dynamically stable} (cf. \cite{LJones2021}).

In a related vein, we call a polynomial $f(x)$ (resp., a pair $(f(x),a)$) \textit{dynamically irreducible} if $f^n(x)$ (resp., $f^{n}(x)-a$) is irreducible for all $n\geq 1$. This is also called \textit{stability} in the literature. Notice that dynamical irreducibility is implied by dynamical monogenicity.
\end{definition}

\begin{remark}
A polynomial is called \textit{post-critically finite} (\textit{PCF}, for short) if the forward orbits of each of its critical points are finite. Theorem \ref{Thm: MainRelativeOrbits} points toward how dynamical monogenicity relates to post-critically finite polynomials. Suppose all iterates of a monic PCF polynomial $f(x)\in\Ocal_K[x]$ are  irreducible. Then, there are only finitely many conditions to check in order to see whether $f(x)$ is dynamically monogenic, or whether $(f(x),a)$ is a dynamically monogenic pair. 

Generally, if $f^n(x)$ is monogenic, then $f^k(x)$ is monogenic for all $k\leq n$. In fact, Proposition 3.4 of \cite{SmithWolske} shows that $f^n(x)$ is monogenic if and only if $f^k(x)$ is monogenic for all $k\leq n$. If $f(x)$ is PCF and dynamically irreducible, then let $l$ be the maximum length of a critical orbit. Theorem \ref{Thm: RelCritPointCrit} in conjunction with \cite[Proposition 3.4]{SmithWolske} shows that if $f^l(x)$ is monogenic, then $f^n(x)$ is monogenic for all $n>1$. Thus, if $f(x)$ is dynamically irreducible, then $f(x)$ is PCF if and only if there exists some $l\geq 1$ such that (from the perspective of Theorem \ref{Thm: MainRelativeOrbits}) the conditions for the monogenicity of $f^l(x)$ imply the monogenicity of $f^n(x)$ for all $n\geq 1$. 
If, conversely, $f(x)\in \mathbb{Z}[x]$ is not PCF, then it is known 
 that for any critical point $\theta$ with infinite forward orbit and any $a$ which is not a postcritical point of $f(x)$, there are infinitely many primes $\gp$ with $v_\gp(f^n(\theta)-a)\ge 1$ for at least one $n\in \mathbb{N}$ (see, e.g., \cite[Lemma12]{BIJJLMRS} for a somewhat stronger assertion). In the context of Uchida's theorem, this means that the squarefreeness conditions of Corollary \ref{Cor: CritsinK}, which need to be verified in order to show monogenicity for all $n\in \mathbb{N}$ simultaneously, will necessarily involve infinitely many primes.
\end{remark}


\section{Results on dynamical irreducibility}\label{Sec: DynIrred}


The following is a statement on when and how often the conditions for dynamical monogenicity obtained in the previous section can be fulfilled. It asserts that, for a PCF polynomial fulfilling certain extra assumptions (notably concerning dynamical irreducibility and $K$-rationality of the critical values), the conditions can be fulfilled as long as there are no local obstructions.
\begin{lemma}
\label{lem:posdensity}
Let $K$ be a number field, and $f(x)\in \Ocal_K[x]$ a PCF polynomial whose leading coefficient is a unit and all of whose critical values are $K$-rational. 
Let $\Omega\subset K$ be the (finite) union of forward orbits of the critical points of $f(x)$.
Assume that $(f(x),a)$ is dynamically irreducible for all $a\in \Ocal_K$ outside of a set of relative density zero (with respect to norm). Furthermore, assume that $f(x)$ has no non-vanishing exceptional primes.
Then the following are equivalent:
\begin{enumerate}[label=\textbf{\arabic*.}, ref=\arabic*]
\item Both of the following hold: \label{LemPos1.}
\begin{enumerate}[label=\textbf{\roman*.}, ref=\roman*]
\item (Non-Vanishing Primes)
For every prime ideal $\mathfrak{p}$ of $\Ocal_K$ with $f'(x) \not\equiv 0 \bmod{\mathfrak{p}}$, there exists at least one residue class modulo $\mathfrak{p}^2$ containing no element of $\Omega$.\footnote{Note that this condition is automatically fulfilled if $|\Omega|\le 3$.} \label{LemPos1.i.}
\item (Vanishing Primes) 
For every prime ideal $\mathfrak{p}$ of $\Ocal_K$ with $f'(x)\equiv 0\bmod \gp$, there exists at least one value $a\in \Ocal_K$ such that $f(x)-a$ is $\gp$-maximal. (See Corollary \ref{Cor: CritsinK} Condition \ref{MainCor2.} for an equivalent condition.) 
\label{LemPos1.ii.}
\end{enumerate}
\item There exists a positive density subset (with respect to norm) of values  $a\in \Ocal_K$ such that $(f(x),a)$ is dynamically monogenic over $K$. \label{LemPos2.}
\end{enumerate}
\end{lemma}
\begin{proof}
The implication \ref{LemPos2.}$\Rightarrow$\ref{LemPos1.} follows directly from Corollary \ref{Cor: CritsinK}, so we assume \ref{LemPos1.}.

For every prime ideal $\mathfrak{p}$ of $\Ocal_K$, let $\Omega_{\mathfrak{p}}$ be the set of all mod-$\mathfrak{p}^2$ reductions of elements of $\Omega$ (excluding those for which the reduction is not defined).
It follows by a standard Chinese remainder argument, that the proportion of $a\in \Ocal_K$ such that $(f(x),a)$ fulfills the conclusion of Corollary \ref{Cor: CritsinK} is bounded below by 
\[\prod_{\mathfrak{p} \text{ vanishing}} \frac{1}{N(\mathfrak{p})^2} \cdot \prod_{\mathfrak{p} \text{ non-vanishing}} \left(1-\frac{|\Omega_{\mathfrak{p}}|}{N(\mathfrak{p})^2}\right).\]  Here, the factor $\prod_{\mathfrak{p}} \frac{1}{N(\mathfrak{p})^2}$ is due to the fact that Condition \ref{MainCor2.} of Corollary \ref{Cor: CritsinK} only depends on the mod-$\gp^2$ residue class of $a$. This factor is positive  since there are only finitely many vanishing primes. Furthermore, by Assumption \ref{LemPos1.i.}, we have $|\Omega_\mathfrak{p}| \le \min\{|\Omega|, N(\mathfrak{p})^2-1\}$ for all non-vanishing primes $\mathfrak{p}$. From this, it follows that $\prod_{\mathfrak{p}} \big(1-\frac{|\Omega_{\mathfrak{p}}|}{N(\mathfrak{p})^2}\big)$ is bounded below by a positive constant factor (arising from the finitely many primes $\mathfrak{p}$ with $N(\mathfrak{p})^2-1 < |\Omega|$) times $\prod_{\mathfrak{p}} \big(1-\frac{|\Omega|}{N(\mathfrak{p})^2}\big)$, the latter product obviously being bounded away from zero. 
In total, the proportion of $a\in \Ocal_K$ such that $(f(x),a)$ fulfills the conclusion of Corollary \ref{Cor: CritsinK} is bounded away from $0$.
\end{proof}




The assumption in Lemma \ref{lem:posdensity} of $(f(x),a)$ being dynamically irreducible for ``most" $a\in \Ocal_K$ is not automatic but can nevertheless be guaranteed in certain key cases.
We collect some results ensuring this property. 
\begin{lemma}
\label{lem:irred}
Let $K$ be a number field, $p$ a prime number, and $f(x)=(f_1\circ\dots\circ f_r)(x)\in K[x]$ be a PCF polynomial such that each $f_i(x)\in K[x]$ is linearly related to $x^p$, i.e., $f_i(x)=(\lambda_i\circ x^p \circ \mu_i)(x)$ with suitable linear polynomials $\lambda_i(x), \mu_i(x)\in K[x]$. 
Then there exists $n_0\in \mathbb{N}$ such that, for all $a\in K$ with $\Gal(f^{n_0}(x)-a/K) \cong \Gal\big(f^{n_0}(x)-t/K(t)\big)$, the pair $(f(x),a)$ is dynamically irreducible. 
\end{lemma}
\begin{proof}
For any field $K\subset F\subset \overline{K}$, let $\textrm{Mon}_{f,F}:=\textrm{Gal}(f(x)-t/F(t))$ denote the monodromy group of $f(x)$ over $F$. It is well-known that $\textrm{Mon}_{f,F}$ embeds as a permutation group into the iterated wreath product $\Mon_{f_r,F}\circ \dots \circ \Mon_{f_1,F}$. Furthermore, for $F=K(\zeta_p)$ the $p$-th cyclotomic field, one clearly has $\Mon_{f_i,F}\cong C_p$, so that in particular $\Mon_{f,F}\le C_p\wr\dots\wr C_p$ is a $p$-group. The assertion is now a straightforward consequence of \cite[Theorem 1.3]{benedetto2023}.
\end{proof}

\begin{remark}
\label{rem:irred}
Note that the assumption on the Galois group of $f^{n_0}(x)-a$ is fulfilled for all but a density zero set of integral values $a\in \Ocal_K$, by a sufficiently strong version of Hilbert's irreducibility theorem, see, e.g., \cite[Theorem 1]{Cohen1979}. 
\end{remark}

The following lemma collects a number of relatively well-known statements on the dynamics of Eisenstein polynomials. It is also useful in ensuring dynamical irreducibility, and eventually dynamical monogenicity. 

\begin{lemma}\label{lem:eisen_dyn}
Let $K$ be a number field and $\gp$ a prime of $K$. If a polynomial $f(x)\in K[x]$ is $\gp$-Eisenstein, then it is $\gp$-maximal. Moreover, the extension generated by $f(x)$ is totally ramified at $\gp$. Conversely, if an extension $L/K$ is totally ramified at $\gp$, then $L/K$ is generated by a polynomial that is $\gp$-Eisenstein. Finally, if $f(x)$ is $\gp$-Eisenstein and has degree greater than 1, then $f^n(x)$ is $\gp$-Eisenstein for all $n\geq 1$.
\end{lemma}

We include a brief proof for completeness.

\begin{proof}
For the first claim, suppose $f(x)$ has degree $d$. Note that $f(x)\equiv x^d\bmod \gp$. However, $x$ does not divide $f(x)$ modulo $\gp^2$. Thus $f(x)$ is $\gp$-maximal. Alternatively, one can use the Montes algorithm/Ore's theorem to see that the principal $x$-polygon is one-sided with slope $-\frac{1}{d}$. This also shows that $\gp$ has ramification index $d$. The fact that iterates of a $\gp$-Eisenstein polynomial are $\gp$-Eisenstein is implied by Lemma 2.3 of \cite{Odoni}, though a direct proof via analysis of the valuations of the coefficients of $f^n(x)$ is possible.

To see that an extension totally ramified over $\gp$ is given by a $\gp$-Eisenstein polynomial, suppose we are in that setup, and let $\gP$ denote the unique prime of $\Ocal_L$ above $\gp$. Let $\pi_\gP\in\Ocal_L$ be an element with $\gP$-adic valuation 1. We see that $K\big(\pi_\gP\big)=L$. Moreover, the minimal polynomial of $\pi_\gP$ over $K$ is $\gp$-Eisenstein. This can be generalized a good deal; the interested reader should consult Proposition 12 in $\S$6 of Chapter III of \cite{SerreLocal}.
\end{proof}


Combining Lemma \ref{lem:eisen_dyn} with Theorem \ref{Thm: MainRelativeOrbits}, we obtain a sufficient criterion for monogenicity of the {\it fields} generated by roots of polynomials $f^{n}(x)-a$, for all $n\in \mathbb{N}$. First, we need a lemma that classifies when units are monogenerators.


\begin{lemma}
\label{lem:inverses}
    Let $R$ be a Dedekind domain, and let $\alpha$ be a root of an irreducible polynomial $f(x)=c_nx^n+c_{n-1}x^{n-1}+\cdots +c_1x+c_0\in R[x]$ with $c_n\in R^\times$. Let $K$ be the field of fractions of $R$, and let $L=K(\alpha)$. Write $\Ocal_L$ for the ring of elements of $L$ which are integral over $R$. The following are equivalent:
    \begin{enumerate}
        \item\label{1} $\alpha\in \Ocal_L^\times$.
        \item\label{2}  $c_0\in R^\times$.
        \item\label{3} $\alpha^{-1}\in R[\alpha]$.
        \item\label{4} $R[\alpha]=R[\alpha^{-1}]$.
    \end{enumerate}
\end{lemma}

\begin{proof}
For \eqref{1}$\implies$\eqref{2}, note that if $\alpha\in\Ocal_L^\times$, then there exists $\alpha^{-1}\in\Ocal_L^\times$ so that $\alpha\cdot \alpha^{-1}=1$. Taking norms we have $\Norm_{L/K}(\alpha)\Norm_{L/K}(\alpha^{-1})=1$. Thus $c_0\in R^\times$. $\checkmark$

For \eqref{2}$\implies$\eqref{3}, we have $c_n\alpha^n+\cdots +c_1\alpha+c_0=0$, so 
\[1=-c_0^{-1}c_n\alpha^n-\cdots -c_0^{-1}c_1\alpha.\ \text{ Thus } \ \alpha\inv = -c_0^{-1}c_n\alpha^{n-1}-\cdots -c_0^{-1}c_1. \ \checkmark\]

For \eqref{3}$\implies$\eqref{4}, we have $\alpha\inv = b_{n-1}\alpha^{n-1}+\cdots +b_1\alpha+b_0$ for some $b_i\in R$. Multiplying by $\alpha^{1-n}$ and subtracting, we obtain
\[\alpha^{-n}-b_0\alpha^{1-n}-b_1\alpha^{2-n}-\cdots -b_{n-3}\alpha^{-2}-b_{n-2}\alpha^{-1}-b_{n-1}=0\]
Since $\alpha^{-1}$ is a unit and since it generates the same extension field as $\alpha$, we see that $x^n-b_0x^{n-1}-\cdots -b_{n-2}x-b_{n-1}$ is the minimal polynomial. \eqref{1}$\implies$\eqref{2} shows that $b_{n-1}\in R^\times$. Thus 
\[b_{n-1}^{-1}\alpha^{1-n}-b_{n-1}^{-1}b_0\alpha^{2-n}-b_{n-1}^{-1}b_1\alpha^{3-n}-\cdots -b_{n-1}^{-1}b_{n-3}\alpha^{-1}-b_{n-1}^{-1}b_{n-2}=\alpha\]
is an $R$-linear expression of $\alpha$ in powers of $\alpha\inv$. Hence $R[\alpha]=R[\alpha^{-1}]$. $\checkmark$

Finally, \eqref{4}$\implies$\eqref{1} is clear since $R[\alpha]=R[\alpha^{-1}]\subset \Ocal_L$. $\checkmark$
\end{proof}


\begin{theorem}
\label{thm:eisen_dynamic}
Let $K$ be a number field. Assume that $f(x)\in \Ocal_K[x]$ is PCF with leading coefficient a unit, and assume that there exists some $\gamma\in \Ocal_K$ which is periodic under $f(x)$. 
Furthermore, assume that all exceptional primes for $f(x)$ are vanishing. 
Given a non-unit $b\in \Ocal_K$, let $K\subset K_1\subset K_2\subset\dots$ be a chain of fields such that $K_n$ is generated over $K$ by a root of $f^n(x)-\frac{1+b\gamma}{b}$.
For any non-vanishing prime $\gp\subset \Ocal_K$ of $f(x)$, let $\gP$ denote a prime extending $\gp$ in the splitting field of $f'(x)$. For any vanishing prime $\gp$ of $f(x)$ not dividing $b$, fix $g(x)\in K[x]$ such that $f(x)-\frac{1+b\gamma}{b}\equiv g(x)^p$ mod $\gp$ (with $p$ the residue characteristic of $\gp$), and let $\gP$ denote a prime extending $\gp$ in the splitting field of $g(x)$.  
Assume that all of the following are fulfilled:
\begin{enumerate}[label=\textbf{\roman*.}, ref=\roman*]
\item $v_{\gp}(b)\le 1$ for all primes $\gp$ of $\Ocal_K$. \label{Eiseni.}
\item For every critical point $\theta$ of $f(x)$, every $k\in \mathbb{N}$, and every prime $\gp$ of $\Ocal_K$ not dividing $b$ which is a non-vanishing prime for $f(x)$, one has $v_{\gP}\big(1+(\gamma-f^k(\theta))b\big)\le e(\gP\mid\gp)$. \label{Eisenii.} 
\item For every prime $\gp$ of $\Ocal_K$ not dividing $b$ which is a vanishing prime for $f(x)$, and for any root $\tau$ of $g(x)$, one has $v_{\gP}\big(f(\tau)-\frac{1+b\gamma}{b}\big)\le e(\gP\mid \gp)$. 
\label{Eiseniii.}
\end{enumerate}
Then $K_n$ is  monogenic over $K$ for all $n\in \mathbb{N}$.
\end{theorem}
\begin{proof}
For $n\in \mathbb{N}$, let $\alpha_n$ be a root of $f^n(x)-\frac{1+b\gamma}{b}$, let $c_n\in \Ocal_K$ be such that $f^n(c_n)=\gamma$ (note that such $c_n$ exists due to $\gamma$ being periodic), and let $\beta_n:=\frac{1}{\alpha_n-c_n}$. In particular, $\beta_n$ is a root of the polynomial $h_n(x):=\big(-bf^n(\frac{1}{x}+c_n) + 1+b\gamma\big)\cdot x^{\deg(f)^n}$ (i.e., the reciprocal of the polynomial $-bf^n(x+c_n) + 1+b\gamma$). We will show that, under the conditions of the theorem, $\Ocal_{K_n}=\Ocal_K[\beta_n]$ for all $n\in \mathbb{N}$.

First, we  make sure that $h_n(x)$ is indeed monic. This amounts to verifying that the constant coefficient of $-bf^n(x+c_n) + 1+b\gamma$ equals $1$, which follows directly from the fact that $f^{n}(c_n)=\gamma$. Next, since the leading coefficient of $f(x)$ is a unit and $v_{\gp}(b)\le 1$, we see that $h_n(x)$ is a $\gp$-Eisenstein polynomial, and hence $\gp$-maximal for all primes $\gp\mid b$ by Lemma \ref{lem:eisen_dyn}. Since $b$ is a non-unit, there is at least one such prime $\gp$ with $v_\gp(b)=1$. In particular, $h_n(x)$ is irreducible.

    Next, pick a prime $\gp$ not dividing $b$ which is a non-vanishing prime for $f(x)$, and consider the polynomial $g_n(x):=f^{n}(x+c_n)-\frac{1+b\gamma}{b}\in \Ocal_{K_\gp}[x]$. 
    Note that $g_n(x) = \widehat{f}^n(x) - \widehat{a}$, where $\widehat{f}(x):=(x-c_n)\circ f\circ (x+c_n)$ and $\widehat{a}:=\frac{1+b\gamma}{b}-c_n$. 
    The critical points of $\widehat{f}(x)$ are exactly the values $\widehat{\theta}:=\theta-c_n$, where $\theta$ runs through the critical points of $f(x)$. In particular, if $\gp$ is not a vanishing prime of $f(x)$, then Assumption \ref{Eisenii.} guarantees that, for every critical point $\theta$ of $f(x)$ and every $k\in \mathbb{N}$,
\[v_{\gP}\left(f^k\left(\widehat{\theta}\right)-\widehat{a}\right) = v_{\gP}\left(f^k(\theta)-\frac{1+b\gamma}{b}\right) = v_{\gP}\left(-\frac{1}{b}
\Big(1+\big(\gamma-f^k(\theta)\big)b\Big)\right)\le e(\gP\mid \gp) .\] 
    Thus Condition \ref{Main1b.} of Theorem \ref{Thm: MainRelativeOrbits} applied with the pair $(\widehat{f}, \widehat{a})$\footnote{Technically, since $\widehat{a}\notin \Ocal_K$, we apply Theorem \ref{Thm: MainRelativeOrbits} with $\Ocal_K$ replaced by its localization at $\gp$, but see Remark \ref{Remark: pmaximalandlocal} justifying this.} shows that $g_n(x)$ is $\gp$-maximal.
   
    Since $g_n(x)$ has constant coefficient $f^n(c_n)-\frac{1+b\gamma}{b}=-\frac{1}{b}$,  which is a unit in $\Ocal_{K_\gp}$, it follows from Lemma \ref{lem:inverses} that  $\gp\nmid  \big[\Ocal_{K_n}:\Ocal_K[\beta_n]\big]$.

    Similarly, if $\gp$ is a vanishing prime not dividing $b$, then, due to Assumption \ref{Eiseniii.}, it follows from Theorem \ref{Thm: MainRelativeOrbits} that $g_n(x)\in \Ocal_{K_\gp}[x]$ is $\gp$-maximal and hence again $\gp\nmid \big[\Ocal_{K_n}:\Ocal_K[\beta_n]\big]$.

    In total, $\Ocal_K[\beta_n]$ is $\gp$-maximal for {\it all} primes $\gp$, i.e., $\beta_n$ is a monogenerator of $\Ocal_{K_n}$ over $\Ocal_K$.
\end{proof}


\section{Applications to specific families of polynomials}\label{sec:appl}

\subsection{Unicritical polynomials}
Arguably the most obvious family of PCF polynomials is given by the monomials $x^d$ (which are unicritical and critically fixed). Investigation of dynamical monogenicity for pairs $(x^d,a)$ with $a \in \ZZ$ has been carried out in \cite{LJones2021} (e.g., Corollary 1). Here, we generalize this widely to unicritical PCF polynomials $f(x)=ux^d+b \in \Ocal_K[x]$. 
Our key observation is that, for polynomials of this shape, dynamical irreducibility follows automatically from the two conditions of Corollary \ref{Cor: CritsinK} and thus does not have to be assumed additionally.
\begin{theorem}
\label{thm:unicrit}
Let $K$ be a number field, $a\in \Ocal_K$, and $f(x)=ux^d+b\in \Ocal_K[x]$ where $u$ is a unit. Then $(f(x),a)$ is dynamically monogenic (and in particular, dynamically irreducible) if and only if both of the following hold:
\begin{enumerate}[label=\textbf{\arabic*.}, ref=\arabic*]
\item For each $n\in \mathbb{N}$ and each prime $\gp$ of $\Ocal_K$ not dividing $d$, one has $v_\gp(f^{n}(0)-a)\le 1$. \label{UniCrit1.}
\item For each prime $\gp$ of $\Ocal_K$ of residue characteristic $p$ dividing $d$, the ratio $\frac{a-b}{u}$ is not congruent to any $p$-th power modulo $\gp^2$. \label{UniCrit2.}
\end{enumerate}
\end{theorem}
\begin{proof}
Note that Condition \ref{UniCrit1.} amounts to Condition \ref{MainCor1.} of Corollary \ref{Cor: CritsinK}, since $0$ is the unique critical point of $f(x)$. Moreover, Condition \ref{UniCrit2.} is equivalent to Condition \ref{MainCor2.} of the same corollary due to the following: for $\gp|(d)$ of residue characteristic $p$ and any $v,c\in \Ocal_K$ such that $v^p\equiv u \bmod \gp$ and $c^p\equiv a-b\bmod \gp$, we have $f(x)-a \equiv (vx^{d/p}-c)^p \bmod \gp$. Choose any root $\theta$ of $vx^{d/p}-c$   
in an extension of $K$. Condition \ref{MainCor2.} of Corollary \ref{Cor: CritsinK} then implies that $f(\theta)-a = u(\frac{c}{v})^p + b-a \not\equiv 0 \bmod \gp^2$, so in particular $\frac{a-b}{u}$ is not congruent modulo $\gp^2$ to the $p$-th power of $\frac{c}{v}$. But on the other hand, the mod-$\gp$ residue class of elements $\gamma\in \Ocal_K$ such that $\gamma^p\equiv \frac{a-b}{u} \bmod \gp$ is unique, and must be the one of $\frac{c}{v}$ by definition. This shows the claimed equivalence.

Since we have verified that our Conditions \ref{UniCrit1.} and \ref{UniCrit2.} amount to those of Corollary \ref{Cor: CritsinK}, we are left with showing that the two conditions here imply that $f^n(x)-a$ is irreducible for all $n\in \mathbb{N}$. We will show this inductively. Let $\alpha_n$ denote a root of $f^n(x)-a$ and set $K_n=K(\alpha_n)$. 
For the base case $n=1$, note that by the so-called Vahlen-Capelli Theorem (see \cite[Theorem 3.1]{Turnwald}), 
$f(x)-a$ being reducible implies either $\frac{a-b}{u}$ being a $p$-th power in $\Ocal_K$ for some prime divisor $p$ of $d$, or  $(-4)\cdot \frac{b-a}{u}$ being a $4$-th power (the latter case being relevant only for $p=2$). In the former case, $\frac{a-b}{u}$ would have to be a $p$-th power modulo $\gp^2$, which was excluded by Condition \ref{UniCrit2.}. 
In the latter case, we may as well assume that $\gp^2$ divides $2$, since if $v_{\gp}(2)=1$, then $(-4)\cdot \frac{b-a}{u}$ being a $4$-th power implies that its $\gp$-adic valuation is a positive multiple of $4$. Thus $v_{\gp}(\frac{b-a}{u})\ge 2$, and we have a contradiction. However, assuming $\gp^2$ divides $2$ implies that $-1$ is a square modulo $\gp^2$. Thus $(-4)\cdot \frac{b-a}{u}$ being a $4$-th power implies $\frac{b-a}{u}$ is a square modulo $\gp^2$, again a contradiction. This concludes the proof of the base case $n=1$.

 Assume now inductively that $f^n(x)-a$ is irreducible. Then it is also monogenic over $K$ by Corollary \ref{Cor: CritsinK}, i.e., $\Ocal_K[\alpha_n] = \Ocal_{K_n}$. To verify that $f^{n+1}(x)-a$ is irreducible, it suffices to show that $f(x)-\alpha_n$ is irreducible over $\Ocal_{K_n}$. Set $\beta_n=\frac{\alpha_n-b}{u}$. Clearly $f(x)-\alpha_n$ is reducible if and only if $x^d-\beta_n$ is. By the Vahlen-Capelli lemma, this can only happen if there exists some prime divisor $p$ of $d$ such that $\beta_n$ or $-\beta_n$ (the latter being needed only in the case that $p=2\mid d$) is a $p$-th power in $\Ocal_{K_n}=\Ocal_K[\alpha_n] = \Ocal_K[x]/(f^n(x)-a)$. 
The latter implies the existence of a polynomial $g(x)\in \Ocal_K[x]$ such that $g(x)^p\equiv \pm (x-b)/u$ mod $(f^n(x)-a)$. But consider this congruence over the residue field $k_{\gp}$ of $\gp$, for a prime $\gp$ of $\Ocal_K$ extending the rational prime $p$. Then both $g(x)^p$ and $f^n(x)-a$ become polynomials in $x^p$, and hence so does the remainder after division of one of them by the other. But clearly this is not the case for the degree-$1$ remainder $\pm(x-b)/u$. This contradiction concludes the proof. 
\end{proof}

Applying Theorem \ref{thm:unicrit}, we obtain a particularly nice result for the family of unicritical polynomials $f(x)=1-x^d\in \mathbb{Z}[x]$. Note that the iterated extensions defined by $f^{n}(x)-a$ behave quite differently from those defined by $x^{d^n}-a$; in particular, the Galois group of the former is in general far from cyclic even after adjoining roots of unity, meaning that the theorem below yields many monogenic extensions with ``interesting" Galois groups.

\begin{theorem}
\label{thm:1minusxq}
Let $d\ge 2$ be an integer, $f(x)=1-x^d$ and $a\in \mathbb{Z}$.
Then the following are equivalent:
\begin{enumerate}[label=\textbf{\arabic*.}, ref=\arabic*]
\item $(f(x),a)$ is dynamically monogenic. \label{Thm1-1.}
\item $f^2(x)-a$ is monogenic. \label{Thm1-2.}
\item Both of the following hold: \label{Thm1-3.}
\begin{enumerate}[label=\textbf{\roman*.}, ref=\roman*]
\item $a$ and $1-a$ are both squarefree, \label{Thm1-3.i.}
\item $(1-a)^{p-1}\not\equiv 1 \bmod{p^2}$ for all prime divisors $p$ of $d$. \label{Thm1-3.ii.}
\end{enumerate}
\end{enumerate}
In particular, the set of values $a\in \mathbb{Z}$ such that $(f(x),a)$ is dynamically monogenic has positive density 
\[\frac{\varphi(d)}{d} \cdot \prod_{\substack{\ell \text{ prime }\\  \ell\nmid d}} \left(1-\frac{2}{\ell^2}\right).\]
\end{theorem}
\begin{proof}
Note that $f(x)=1-x^d$ is a unicritical PCF polynomial with critical orbit $\{0,1\}$ of length $2$. 
From 
Theorem \ref{thm:unicrit}, dynamical monogenicity of $(f(x),a)$ is equivalent to the following two conditions being fulfilled for all primes $p$:
\begin{enumerate}[label=\textbf{\Alph*.}, ref=\Alph*]
\item If $p$ is a non-vanishing prime for $f(x)$ (i.e., a prime not dividing $d$), then $p^2$ does not divide any of the integers $a$ and $1-a$. \label{ItemA}
\item If $p$ is a vanishing prime, then $1-a$ is not a $p$-th power modulo $p^2$. \label{ItemB}
\end{enumerate}
Note that \ref{ItemB} implies the following:
\begin{enumerate}[label=\textbf{C.}, ref=C]
\item Neither $a$ nor $1-a$ is divisible by $p^2$. \label{ItemC}
\end{enumerate}
Hence, dynamical monogenicity is equivalent to all Conditions \ref{ItemA}, \ref{ItemB}, and \ref{ItemC} being fulfilled. But \ref{ItemA} and \ref{ItemC} together are equivalent to Condition \ref{Thm1-3.}\ref{Thm1-3.i.}. On the other hand, note that the {\it non-zero} $p$-th powers in $\mathbb{Z}/p^2\mathbb{Z}$ are exactly the $p$-th powers which are coprime to $p$, i.e., the elements of $(\mathbb{Z}/p^2\mathbb{Z})^\times \cong C_{p(p-1)}$ whose multiplicative order divides $p-1$. Hence \ref{ItemB} and \ref{ItemC} together are equivalent to \ref{Thm1-3.}\ref{Thm1-3.ii.}

In total, we have obtained that the dynamical monogenicity of $(f(x),a)$ is equivalent to $a$ fulfilling both Conditions \ref{Thm1-3.}\ref{Thm1-3.i.} and \ref{Thm1-3.}\ref{Thm1-3.ii.} of the theorem. On the other hand, Conditions \ref{Thm1-3.}\ref{Thm1-3.i.} and \ref{Thm1-3.}\ref{Thm1-3.ii.} are exactly the conditions obtained from setting $N=2$ in Corollary \ref{Cor: CritsinK}, whence \ref{Thm1-2.} implies \ref{Thm1-3.}. The implication \ref{Thm1-1.}$\Rightarrow$\ref{Thm1-2.} is trivial.

Lastly, the density assertion follows by a Chinese remainder argument as in the proof of Lemma \ref{lem:posdensity}. Explicitly, we notice that Condition \ref{Thm1-3.}\ref{Thm1-3.i.} excludes exactly two residue classes modulo $\ell^2$ for all primes $\ell$ not dividing $d$. Whereas, for primes $p|d$, $1-a$ satisfies \ref{Thm1-3.}\ref{Thm1-3.ii.} in $p(p-1)$ of the possible $p^2-1$ nonzero residue classes modulo $p^2$, and \ref{Thm1-3.}\ref{Thm1-3.i.} additionally excludes only $1-a\equiv 0$ mod $p^2$. Thus, a total of exactly $p(p-1)$ out of $p^2$ residue classes are allowed.
\end{proof}

For pairs $(f(x),a)$ with $f(x)=1-x^d$ and $a$ not necessarily integral,
we can use Theorem \ref{thm:eisen_dynamic} to obtain results on dynamic monogenicity for the {\it fields} generated by roots of $f^{n}(x)-a$ (rather than monogenicity of the polynomials themselves). 

Concretely, we have the following:
\begin{theorem}\label{Thm: RadFields}
Let $d\in \mathbb{N}$ and $b\in \mathbb{N}$ be positive integers such that all of the following hold:
\begin{enumerate}[label=\textbf{\roman*.}, ref=\roman*]
\item $b$ and $b-1$ are squarefree. \label{RadFieldi.}
\item For all primes $p$ which divide $d$, but not $b$, one has $(1-1/b)^{p-1}\not\equiv 1 \bmod{p^2}$. \label{RadFieldii.}
\end{enumerate}
Let $f(x)=1-x^d$, and write $K_n$ for an extension of $\mathbb{Q}$ generated by a root of $f^n(x)-1/b$. Then $K_n$ is monogenic for all $n\in \mathbb{N}$.
\end{theorem}

\begin{proof}
This follows from Theorem \ref{thm:eisen_dynamic} via using the periodic point $\gamma=0$. We only need to verify that Assumptions \ref{RadFieldi.} and \ref{RadFieldii.} here imply Assumptions \ref{Eiseni.}, \ref{Eisenii.}, and \ref{Eiseniii.} of Theorem \ref{thm:eisen_dynamic}. Indeed, since the set of postcritical points of $f(x)$ is $\{0,1\}$, the squarefreeness of $b$ and of $b-1$ amounts to Conditions \ref{Eiseni.} and \ref{Eisenii.} of Theorem \ref{thm:eisen_dynamic}, respectively. Further, Condition \ref{RadFieldii.} here amounts to Condition \ref{Eiseniii.} of Theorem \ref{thm:eisen_dynamic}, as explained in the proof of Theorem \ref{thm:1minusxq}.
\end{proof}

Unicritical PCF polynomials defined over $\QQ$ only produce critical orbits with fixed points or $2$-cycles, cf.\ \cite[Theorems 2.1 and 2.3]{anderson_cubic_2020}. 
In order to construct critical orbits containing an $n$-cycle, we find quadratic polynomials $h(x)=x^2+\alpha_n$, where $\alpha_n$ is defined by $h^n(\alpha_n)=0$ and $h^k(\alpha_n)\neq 0$ for $0\leq k <n$. The fields of definition for $\alpha_n$ are given by the Gleason polynomials for order $n$, and $\alpha_n$ is a Misiurewicz point of period $(1,n)$. (See \cite{Buff},\cite{GokselUnicrit},\cite{HutzTowsley} for additional exposition.)

\begin{example}\label{Example: Quadratic 3-cycle}
Let $K=\QQ[\theta]$ be the number field defined by $\theta^3-\theta^2+1=0$ (note that $\Ocal_K=\ZZ[\theta]$), and let $f(x)=x^2+\theta-1$. Then $f(x)$ is PCF, since $f(0)=\theta-1$, $f(\theta-1)=\theta^2-\theta$, and $f(\theta^2-\theta)=0$, and hence has a $3$-cycle as its critical orbit. 

Let $a \in \Ocal_K$ such that $\{a, a-\theta+1, a-\theta^2+\theta\}$ are squarefree, and that $a \not\equiv \theta-1 \bmod 2$. Note that $2$ is inert in $\Ocal_K$, and is the only vanishing prime. Then by Theorem \ref{thm:unicrit},  $(f(x), a)$ is dynamically monogenic over $K$.
\end{example}


\subsection{Chebyshev polynomials}
In this section, we determine all dynamically monogenic pairs $(T_d(x),a)$ where $a\in \mathbb{Z}$ and $T_d(x)$ is the (normalized) degree $d$ Chebyshev polynomial, thereby generalizing \cite[Theorem 1.2]{GassertRadical} which deals with the case where $d$ is prime, and characterize those pairs which generate totally real fields. Using the notation and results from  \cite{SilvermanDynamical}, the normalized Chebyshev polynomials $T_d(x)\in \ZZ[x]$ have the following properties: 
\begin{enumerate}
    \item They satisfy the recurrence $T_{d+1}(x)=xT_d(x)-T_{d-1}(x)$, with $T_0(x)=2, \ T_1(x)=x$.
    \item They are monic.
    \item They commute: $T_d\big(T_e(x)\big)=T_{de}(x)=T_e\big(T_d(x)\big)$. Note, composition gives $T_d^n(x)=T_{d^n}(x)$, so iterates of Chebyshev polynomials are also Chebyshev polynomials.
\item Their derivatives are $T'_d(x)=dU_{d-1}(x)$, where $U_{d+1}(x)=xU_d(x)-U_{d-1}(x)$ with $U_{-1}=0$ and $U_0(x)=1$. The $U_i(x)$ are the normalized Chebyshev polynomials of the second kind. Together $T_d(x)$ and $U_{d-1}(x)$ satisfy the quadratic relation
\[
T_d(x)^2=4+(x^2-4)U_{d-1}(x)^2.
\]
\end{enumerate}
Using this last relation, we see the $T_d(x)$ are PCF. Indeed, if $U_{d-1}(\alpha)=0$, then $T_d(\alpha)=\pm 2$, and $T_d(\pm2)^2=4$. Using the recurrence, we have $T_d(2)=2$, $T_d(-2)=(-1)^d2$. Further, the roots of $T_d(x)$ and $U_d(x)$ are real, belong to the interval $[-2,2]$, and are simple. They are given, respectively, by 
\[2\cos\left(\frac{(2k-1)\pi}{2d}\right)=\zeta_{2d}^{2k-1}+\zeta_{2d}^{1-2k} \quad \text{ and }\quad 2\cos\left(\frac{k\pi}{2(d+1)}\right)=\zeta_{2(d+1)}^{k}+\zeta_{2(d+1)}^{-k},\]
where $1\leq k \leq d$, and $\zeta_{n}$ is a primitive $n$-th root of unity.
The main result of this section is the following.
\begin{theorem}
\label{thm:cheby}
Let $d\ge 2$, and let $T_d(x)$ be the normalized Chebyshev polynomial of degree $d$. Then for $a\in \mathbb{Z}$, the pair $(T_d(x),a)$ is dynamically monogenic if and only if all of the following hold.
\begin{enumerate}[label=\textbf{\roman*.}, ref=\roman*]
\item For all primes $p$ not dividing $d$, one has $a\not\equiv \pm 2 \bmod {p^2}$. \label{Chebyi.}
\item For all primes $p|d$, one has $T_p(a)\not\equiv a \bmod {p^2}$. \label{Chebyii.}
\end{enumerate}
\end{theorem}

We prove Theorem \ref{thm:cheby} by first assuming a pair is dynamically irreducible and applying the main theorem, then showing that the conditions required for the pair to be dynamically monogenic are sufficient to imply they are dynamically irreducible.

\begin{lemma}\label{Lem: ChebyMonoIfIrred}
    Let $d\ge 2$ and $a\in \ZZ$ such that $(T_d(x),a)$ is dynamically irreducible. Then $(T_d(x),a)$ is dynamically monogenic if and only if 
    \begin{itemize}
        \item for all primes $p$ not dividing $d$, one has $a\not\equiv \pm 2 \bmod {p^2}$, and
        \item for all primes $p|d$, $T_p(a)\not\equiv a \bmod {p^2}$.
    \end{itemize}
\end{lemma}

\begin{proof}
    Suppose first that $p$ is a prime not dividing $d$. Then $T_d'(x)=dU_{d-1}(x)$ is not identically $0$ in $\FF_p$, so $p$ is a non-vanishing prime. Furthermore, every root of $T_d'(x)$ is of the form $\zeta^k+\zeta^{-k}$ for some $2d$-th root of unity $\zeta$. These elements are well-known to be monogenerators for the field they generate over $\mathbb{Q}$, meaning in particular that their minimal polynomial remains separable modulo all primes not dividing the field discriminant. But prime divisors of the field discriminant have to ramify in $\mathbb{Q}(\zeta_{2d})$, i.e., have to divide $d$. We have therefore shown that all non-vanishing primes are in fact non-exceptional.
    
    The PCF properties of $T_d(x)$ show that $T_d^n(\theta)\in \{-2,2\}$ for any critical point $\theta$, so we can apply Corollary \ref{Cor: CritsinK} for the non-vanishing primes. Thus $T_d^n(x)-a$ is $p$-maximal if and only if $p^2$ does not divide $-2-a$ or $2-a$.

    Next, suppose $p$ is a prime dividing $d$. We have $T_d'(x)\equiv 0 \bmod p$, so $p$ is a vanishing prime. Let $v_p(d)=k$, and note $T_{p^k}(x)\equiv x^{p^k} \bmod p$.
    Let $e=d/p^k$, so $T_d(x)=T_p^k(T_e(x))$.
    In particular 
    \[
    T_d(x)-a \equiv T_e(x)^{p^k}-a \equiv (T_e(x)-a)^{p^k} \bmod p.
    \]To apply Corollary \ref{Cor: CritsinK}, we need to consider $T_d(\tau)- a$, where $\tau\in \overline{\QQ}$ is a root of $T_e(x)-a$. More precisely, $T_d(x)-a$ is $p$-maximal if and only if $T_d(\tau)-a = T_{p}^k(a)-a\in \mathbb{Z}$ is not divisible by $p^2$.
    We reduce to $T_p(a)$ by applying Proposition 3.4 of \cite{GassertRadical}: for $p$ an odd prime, 
    \[
    T_p(b)\equiv T_p (c) \bmod {p^2} \iff b\equiv c \bmod p
    \]
    with $b=T_{p^{k-1}}(a)$ and $c=a$. The extension to $p=2$ is immediate from $T_2(x)=x^2-2$.    
\end{proof}


\begin{remark}
\label{rem:modp2values}
As seen in the above proof, the condition $T_p(a)\not\equiv a\bmod p^2$ is in fact equivalent to $a$ not being in the image of $T_p(x)$ modulo $p^2$, which leaves a total of $p^2-p$ residue classes modulo $p^2$. In particular, Theorem \ref{thm:cheby} then implies that the set of $a\in \mathbb{Z}$ with $(T_d(x),a)$ dynamically monogenic is of density 
\[
\begin{cases}
 \frac{\varphi(d)}{d}\cdot \prod_{\ell}(1-\frac{2}{\ell^2}) & \text{if } d \text{ is even} \\ \frac{\varphi(d)}{d}\cdot \frac{3}{4}\prod_{\ell\ne 2}(1-\frac{2}{\ell^2}) & \text{if } d \text{ is odd}   
\end{cases}\]
where in both cases the product is over primes $\ell$ not dividing $d$.
\end{remark}

To obtain Theorem \ref{thm:cheby}, it remains to prove that the dynamical irreducibility assumption made in Lemma \ref{Lem: ChebyMonoIfIrred} follows from Condition \ref{Chebyii.} of Theorem \ref{thm:cheby}. Note this generalizes \cite{GassertRadical} even in the prime degree case. 
\begin{lemma}
Assume that $a\in \mathbb{Z}$ is such that $T_p(a)\not\equiv a \bmod {p^2}$ for all primes $p|d$. Then $(T_d(x),a)$ is dynamically irreducible.
\end{lemma}
\begin{proof}
Since $T_{d}^n(x) = T_{d^n}(x)$, it suffices to deduce irreducibility of $T_d(x)-a$ for any $d$. By the main theorem of \cite{Turnwald}, the latter is equivalent to $a$ not being a rational value of any polynomial $T_p(x)$ for any prime $p$ dividing $d$. Further, in the case that $4|d$ there is the additional requirement that $a$ is not of the form $-4c^4+8c^2-2$ with $c\in \mathbb{Q}$. (This latter condition is equivalent to the irreducibility of $T_4(x)-a$.)
As already noted, the condition $T_p(a)\not\equiv a \bmod {p^2}$ is equivalent to $a$ not being a value of $T_p(x)$ modulo $p^2$, which implies in particular that $T_p(x)-a$ has no rational root for any $p|d$. Furthermore, in the case where $2|d$, we have the requirement $a\not\equiv 2,3$ mod $4$, which in particular prevents the shape $a=-4c^4+8c^2-2$. This concludes the proof.
\end{proof}

We conclude this section with considerations about totally real monogenic extensions.

\begin{theorem}\label{Thm:RealMono}
Let $d\ge 3$ and $a\in \mathbb{Z}$. The polynomial $T_d(x)-a$ generates a totally real monogenic field if and only if one of the following holds:
\begin{enumerate}[label=\textbf{\arabic*.}, ref=\arabic*]
\item $a=1$ and $d=2^b3^c$ with $b,c\ge 0$; \label{Real1.}
\item $a=-1$ and $d=3^c$ with $c\ge 1$; \label{Real2.}
\item $a=0$ and $d=2^b$ with $b\ge 2$. \label{Real3.}
\end{enumerate}
\end{theorem}
\begin{proof}
Since $d\ge 3$, the critical values of $T_d(x)$ are exactly $\pm 2$. It is well-known that all of the roots of $T_d(x)$ are real, and since the number of real roots of $T_d(x)-a$ can only change as $a$ passes a critical value, it follows that the splitting field of $T_d(x)\pm 1$ is real as well. On the other hand, a polynomial $f(x)-a$ of degree $\ge 3$ can never have all of its roots real as $a\to \pm \infty$. Thus $T_d(x)-a$ cannot generate a totally real field for $|a|>2$.  Moreover, $T_d(x)-a$ is reducible for $a=\pm 2$ a critical value. This leaves only the values $a\in \{-1,0,1\}$.

 The polynomial $T_d(x)$ is odd when $d$ is odd, hence $T_d(0)=0$. This means that $T_d(x)$ is reducible for any odd integer $d\ge 3$, and thus also for any $d$ with at least one odd prime factor, due to the relation $T_d(T_e(x)) = T_{de}(x)$.

Moreover, if $d\equiv \pm 1 \bmod 6$, then $T_d(1)=1$, since the recurrence $T_{d+1}(x)=xT_{d}(x)-T_{d-1}(x)$ has period $6$ when $x=1$. Again, since $T_d(x)$ is odd when $d$ is odd, one also gets $T_d(-1)=-1$. This means that if $p>3$ is prime, the polynomials $T_p(x)-a$ are also reducible  for $a=\pm 1$, and hence, so are the polynomials $T_d(x)-a$ for any $d$ divisible by at least one prime $p>3$. Lastly, $T_2(x)+1=x^2-1$ is reducible, and thus, so is $T_d(x)+1$ for every even integer $d$. This leaves the values $(d,a)=(2^b3^c, 1)$ for $b,c\ge 0$ arbitrary, $(d,a)=(3^b,-1)$, and $(d,a)=(2^b, 0)$.

 For Cases \ref{Real1.} and \ref{Real2.}, monogenicity follows from Theorem \ref{thm:cheby} since $T_2(x)=x^2-2$ maps $1$ to $-1$, and $T_3(x)=x^3-3x$ maps $1$ to $-2$ and $-1$ to $2$.  Finally, for Case \ref{Real3.}, 
 monogenicity follows since $T_2(x)=x^2-2$ does not fix $0$ modulo $4$.
\end{proof}

Note that the polynomials $T_d(x)-a$ in Theorem \ref{Thm:RealMono} all generate abelian extensions of $\mathbb{Q}$; indeed, as a special case of \cite[Theorem 13]{AndrewsPetsche}, the dynamical Galois group $\textrm{Gal}(T_{d^\infty}(x)-a/\mathbb{Q})$ is abelian exactly for $a\in \{-2,-1,0,1,2\}$. A wider variety of totally real monogenic fields (in particular, nonabelian ones) can be obtained via resorting to non-integral specialization values: 
\begin{theorem}
\label{Thm:RealMono2}
Let $d$ be an integer and $b<-1$ a negative integer such that all of the following are fulfilled:
\begin{enumerate}[label=\textbf{\roman*.}, ref=\roman*]
\item $b$ and $1+4b$ are squarefree. \label{Real2i.}
\item For all prime divisors $p$ of $d$ which do not divide $b$, one has 
$T_p(2+\frac{1}{b})\not\equiv 2+\frac{1}{b} \bmod p^2$. \label{Real2ii.}
\end{enumerate}
Then the field generated over $\mathbb{Q}$ by a root of $T_d(x)-(2+\frac{1}{b})$ is monogenic and totally real.
\end{theorem}
\begin{proof}
Monogenicity is a straightforward consequence of Theorem \ref{thm:eisen_dynamic} with $\gamma=2$ (which is a fixed point of $T_d(x)$), upon noting again that the postcritical points of $T_d(x)$ are exactly $\pm 2$ and that Condition \ref{Real2ii.} here is the translation of Condition \ref{Eiseniii.} in Theorem \ref{thm:eisen_dynamic}, as already shown in the proof of Theorem \ref{thm:cheby}. That the field is totally real is immediate from $|2+\frac{1}{b}| < 2$ (due to $b$ being negative), as noted in the proof of Theorem \ref{Thm:RealMono}.
\end{proof}


\subsection{PCF compositions of several polynomials}
We describe a family of PCF polynomials arising as compositions of two polynomials in a non-trivial way. In particular, this family is of a different shape than any of the families considered in the previous sections. These polynomials are defined over cyclotomic fields and can be constructed to have any number of critical points. Let $d$ be a positive integer, write $\zeta$ for a primitive $d$-th root of unity, and $\mu$ for a primitive $(2d-1)$-th root of unity if $d$ is odd, or a primitive $(4d-2)$-th root if $d$ is even. Let $K=\QQ(\mu)$. Then 
\[
f(x)=x^{2d}+2\mu x^d=x^d(x^d+2\mu) = (x^2+2\mu x)\circ x^d\qquad \text{with} \qquad
f'(x)=2dx^{d-1}(x^d+\mu)
\]has $d+1$ critical points: $0$ and $(-1)^d\mu^2\zeta^k$ for each $0\leq k<d$. Each non-zero critical point is mapped by $f(x)$ to $-\mu^2$, which is fixed by $f(x)$, while $0$ is mapped to $0$. For each prime $p$ dividing $2d$ and each $a \in \Ocal_K$, $\mu$ is a $p$-th power in $\Ocal_K$ because $(p,2d-1)=1$, and we have
\[
f(x)-a \equiv \left(x^{\frac{2d}{p}} + 2\mu^{\frac 1p}x^{\frac{d}{p}}-a\right)^p \bmod p.
\] When $d=2^k$, this becomes a power of a linear polynomial in $\Ocal_K[x]$, so we can apply Lemma \ref{lem:irred} to find values of $a$ where $(f(x),a)$ is dynamically irreducible, and Corollary \ref{Cor: CritsinK} to restrict to those that are dynamically monogenic.

\begin{theorem}\label{Thm:ManyCritsExample}
For $k\geq 1$, let $d=2^k$, $\mu$ a primitive $(4d-2)$-th root of unity, $K=\QQ(\mu)$ and $f(x)=x^{2d}+2\mu x^d$. 
Then the following hold:
\begin{enumerate}[label=\textbf{\arabic*.}, ref=\arabic*]
\item The set of all $a \in \Ocal_K$ such that $(f(x),a)$ is dynamically irreducible, $(a)$ and $(a+\mu^2)$ are squarefree, and $f(a)-a$ is not divisible by $\gp^2$ for any prime ideal $\gp$ above $2$, is a positive density subset of $\Ocal_K$. \label{ManyCrit1.}
\item For all $a\in \Ocal_K$ with the above properties,  $(f(x),a)$ is dynamically monogenic over $K$. \label{ManyCrit2.}
\end{enumerate}
\end{theorem}
\begin{proof}
Since $f(x)$ is a composition of quadratic polynomials, it follows from Lemma \ref{lem:irred} and Remark \ref{rem:irred} that $(f(x),a)$ is dynamically irreducible for all but a density zero set of values $a\in \Ocal_K$.
To apply Corollary \ref{Cor: CritsinK} and Lemma \ref{lem:posdensity}, note that the prime ideals above $2$ in $\Ocal_K$ are vanishing primes, and since for all primes $\gp$ not dividing $2$ the polynomial $x^d+\mu$ remains separable modulo $\gp$, all non-vanishing primes are non-exceptional.
The PCF properties of $f(x)$ and the squarefree conditions on $(a)$ and $(a+\mu^2)$ ensure that the non-vanishing primes condition holds, and here clearly Condition \ref{LemPos1.i.} of Lemma \ref{lem:posdensity} holds (with $|\Omega|=2$).

Shifting to the vanishing primes, the primes above $2$, the field $K$ is also the $(2^{k+1}-1)$-th cyclotomic field, so $(2)$ factors into $\phi(2^{k+1}-1)/(k+1)$ distinct prime ideals, each with inertia degree $k+1$. Let $\gp$ be one such prime. The assumption that $f(a)-a$ is not divisible by $\gp^2$ is precisely what the vanishing prime condition of Corollary \ref{Cor: CritsinK} requires. In order for Condition \ref{LemPos1.ii.} of Lemma \ref{lem:posdensity} to hold, it suffices to find one value $a\in \Ocal_K$ for which $f(a)-a$ has this property. But note that, if $\gp$ divides $(a)$ and $(a)$ is squarefree, then $f(a)-a$ is not divisible by $\gp^2$, since $v_\gp\big(f(a)\big)=2^k+1$ and $v_\gp(a)=1$. Hence the assumption holds if $\gp$ strictly divides $a$. In particular, $a=2$ works for all such $\gp$.

It thus follows from Lemma \ref{lem:posdensity} that a positive density of values $a\in \Ocal_K$ fulfill all the conditions in \ref{ManyCrit1.} and hence render $(f(x),a)$ dynamically monogenic.
\end{proof}

\begin{question}
Is dynamical irreducibility in fact implied by the remaining assumptions of Theorem \ref{Thm:ManyCritsExample}? We are not aware of a counterexample, although our criteria (such as Theorem \ref{thm:unicrit}) are not applicable to compositions of arbitrary quadratic polynomials, as shown by the example in Remark \ref{Rem: CondiesnotIrred}.
\end{question}


\section*{Acknowledgements}
The authors are especially grateful to Sebastian Bozlee for the insightful suggestion that Uchida's criterion could be used to create a general connection between critical points and monogenicity.


\bibliography{Bibliography}
\bibliographystyle{alpha}




\end{document}